\documentclass[10pt,letterpaper, boxed, cm]{amsart}

\usepackage[margin=1in]{geometry}
\usepackage{amsmath,amsthm,amsfonts,amssymb,amscd}
\usepackage{lastpage}
\usepackage{fancyhdr}
\usepackage{mathrsfs}
\usepackage{xcolor}
\usepackage{graphicx}
\usepackage{listings}
\usepackage{hyperref}
\usepackage{enumitem}
\usepackage{relsize}
\usepackage{tikz-cd}
\usepackage{mdwlist}
\usepackage{multicol}
\usepackage{stmaryrd}
\usepackage{float}
\usepackage{adjustbox}
\usepackage{tikz}
\usetikzlibrary{matrix, calc, arrows}

\hypersetup{
    colorlinks=true, 
    linktoc=all,     
    linkcolor=blue,  
}

\setlist[enumerate]{label=(\alph*)}

\usetikzlibrary{graphs,decorations.pathmorphing,decorations.markings}
\tikzcdset{scale cd/.style={every label/.append style={scale=#1},
        cells={nodes={scale=#1}}}}

\newcommand{\Z}{\mathbb{Z}}
\newcommand{\R}{\mathbb{R}}

\newcommand{\Q}{\mathbb{Q}}
\newcommand{\N}{\mathbb{N}}

\newcommand{\Hom}{\operatorname{Hom}}
\newcommand{\Aut}{\operatorname{Aut}}
\newcommand{\tr}{\operatorname{tr}}
\newcommand{\rk}{\operatorname{rk}}
\newcommand{\inv}{^{-1}}

\newcommand{\End}{\operatorname{End}}

\newcommand{\Res}{\operatorname{Res}}
\newcommand{\Ind}{\operatorname{Ind}}
\newcommand{\Inf}{\operatorname{Inf}}

\newcommand{\id}{\operatorname{id}}

\newcommand{\im}{\operatorname{im}}

\newcommand{\calB}{\mathcal{B}}
\newcommand{\calC}{\mathcal{C}}

\newcommand{\calE}{\mathcal{E}}
\newcommand{\calF}{\mathcal{F}}

\newcommand{\calH}{\mathcal{H}}

\newcommand{\calM}{\mathcal{M}}

\newcommand{\calO}{\mathcal{O}}
\newcommand{\calP}{\mathcal{P}}

\newcommand{\calT}{\mathcal{T}}

\newcommand{\calX}{\mathcal{X}}

\newcommand{\Syl}{\operatorname{Syl}}

\newcommand{\triv}{\mathbf{triv}}

\newcommand{\catmod}{\mathbf{mod}}

\newcommand{\Dim}{\operatorname{Dim}}

\newtheorem{theorem}{Theorem}[section]
\newtheorem{lemma}[theorem]{Lemma}
\newtheorem{prop}[theorem]{Proposition}
\newtheorem{corollary}[theorem]{Corollary}

\newtheorem*{theorem*}{Theorem}

\theoremstyle{remark}
\newtheorem{remark}[theorem]{Remark}

\theoremstyle{definition}
\newtheorem{definition}[theorem]{Definition}

\begin{document}
    \title{The classification of endotrivial complexes}
    \author{Sam K. Miller}
    \address{university of California, Santa Cruz, Department of Mathematics} 
    \email{sakmille@ucsc.edu} 

    \begin{abstract}
        Let $G$ be a finite group and $k$ a field of prime characteristic $p$. We give a complete classification of endotrivial complexes, i.e. determine the Picard group $\calE_k(G)$ of the tensor-triangulated category $K^b({}_{kG}\triv)$, the bounded homotopy category of $p$-permutation modules, which Balmer and Gallauer recently considered in \cite{BG23}. For $p$-groups, we identify $\calE_k(-)$ with the rational $p$-biset functor $CF_b(-)$ of Borel-Smith functions and recover a short exact sequence of rational $p$-biset functors constructed by Bouc and Yal\c{c}in. As a consequence, we prove that every $p$-permutation autoequivalence of a $p$-group arises from a splendid Rickard autoequivalence. Additionally, we give a positive answer to a question of Gelvin and Yal\c{c}in in \cite{GeYa21}, showing the kernel of the Bouc homomorphism for an arbitrary finite group $G$ is described by superclass functions $f: s_p(G) \to \Z$ satisfying the oriented Artin-Borel-Smith conditions.
    \end{abstract}

    \maketitle

    \section{Introduction}

    Let $G$ be a finite group and $k$ a field of prime characteristic $p$. In \cite{SKM23} we considered the notion of an endotrivial complex of $p$-permutation $kG$-modules, i.e. an invertible object in the tensor-triangulated category $K^b({}_{kG}\triv)$, the bounded homotopy category of finitely generated $p$-permutation $kG$-modules. This tt-category has been studied in great detail recently by Balmer and Gallauer \cite{BaGa23, BG22, BG23}; in particular, the authors deduce its Balmer spectrum. Endotrivial complexes are analogues of endotrivial modules, the invertible objects of the stable module category ${}_{kG}\underline{\catmod}$, which have been studied extensively by representation theorists (see \cite{Ma18} for a detailed account). In \cite{SKM24a}, we introduced notions of relatively endotrivial chain complexes, chain complex-theoretic analogues of relatively endotrivial $kG$-modules, which were first introduced by Lassueur \cite{CL11}.

    Our main result of this paper is a complete classification of endotrivial complexes. We determine the parametrizing group $\calE_k(G)$, i.e. the Picard group of $K^b({}_{kG}\triv)$, see Definition \ref{def:grpofetrivs}.

    \begin{theorem*}[Corollary \ref{thm:borelsmithfornonpgroups}] \label{thm:mainresult}
        Let $G$ be a finite group. We have a split exact sequence of abelian groups \[0 \to \Hom(G,k^\times) \to \calE_k(G) \xrightarrow{h} CF_b(G,p) \to 0.\] Here, $CF_b(G,p)$ denotes the additive group of Borel-Smith superclass functions $f: s_p(G) \to \Z$ valued on $p$-subgroups (see Definition \ref{def:bsf}), and $h$ denotes the h-mark homomorphism (see Proposition \ref{hmarkhoms}). In particular, if $G$ is a $p$-group, we have an isomorphism $h: \calE_k(G)\cong CF_b(G)$.
    \end{theorem*}
    
    We apply much of the theory developed in \cite{SKM23, SKM24a} to derive this classification and beyond, answering questions concerning endotrivial complexes, the Bouc homomorphism (see Definition \ref{def:bouchom}), and splendid Rickard equivalences. A key structure in this study is the Dade group $D_k(G)$ of a finite group. First introduced by Dade in \cite{Da78} and \cite{Da78b} for the case when $G$ is a $p$-group, $D_k(G)$ parameterizes the ``capped'' endopermutation $kG$-modules. $D_k(G)$ was generalized to all finite groups $G$ by Linckelmann and Mazza in \cite{LiMa09}, by Lassueur in \cite{CL13}, and by Gelvin and Yal\c{c}in in \cite{GeYa21} via separate techniques. In Lassueur's construction, the generalized Dade group $D_k(G)$ parameterizes the ``strongly capped'' endo-$p$-permutation $kG$-modules. We define the $kG$-module $V(\calF_G)$ as follows: \[V(\calF_G) := \bigoplus_{P \in [s_p(G)\setminus \Syl_p(G)]} k[G/P].\] Then, $D_k(G)$ is realized as the subgroup generated by endo-$p$-permutation modules of the group $T_{V(\calF_G)}(G)$ of $V(\calF_G)$-endotrivial $kG$-modules (see Definition \ref{dadegroupgeneralizations}). These modules are closely tied to $V(\calF_G)$-endosplit-trivial complexes, as defined in \cite{SKM24a}, as a $V(\calF_G)$-endosplit-trivial complex is equivalently a shifted endosplit $p$-permutation resolution of a relatively $V(\calF_G)$-endotrivial module. This observation gives us a short exact sequence, which is detailed in Theorem \ref{kernelofhomology}, connecting endotrivial complexes and the generalized Dade group. If $G$ is a $p$-group, the story simplifies considerably and we recover a short exact sequence constructed by Bouc and Yal\c{c}in in \cite{BoYa06}. 

    \begin{theorem*}[Theorem \ref{bigtheorem}]
        Let $G$ be a $p$-group. We have an isomorphism of short exact sequences of rational $p$-biset functors:

        \begin{figure}[H]
            \centering
            \begin{tikzcd}
                0 \ar[r] & \calE_k(G) \ar[r] \ar[d, "h"] & \calE_k^{V(\calF_G)}(G) \ar[r, "\calH"] \ar[d, "h"] & D^\Omega(G) \ar[r] \ar[d, "="] & 0\\
                0 \ar[r] & CF_b(G) \ar[r]  & CF(G) \ar[r, "\Psi"] & D^\Omega(G) \ar[r] & 0
            \end{tikzcd}
        \end{figure}

        Here, the map $\Psi: CF(G) \to D^\Omega(G)$ is the Bouc homomorphism, see Definition \ref{def:bouchom}.
    \end{theorem*}

    Next, we transport the superclass function and Borel-Smith function biset functor structure to $\calE_k(G)$ and $\calE_k^{V(\calF_G)}(G)$ respectively in Section \ref{sec:bisetstructure}. In particular, we verify that the rational $p$-biset functor structure of $\calE_k(-)$ is what one might expect. Interestingly, the induction operation does not coincide with the usual tensor induction of chain complexes, but rather, the topological norm map (see \cite[Section 4.1]{Be982} or \cite[Section A.4]{HHR16}). This fits in with the topological nature of endotrivial complexes - morally, they arise as the Bredon cohomology of representation spheres. On the other hand, the biset functor structure of $\calE_k^{V(\calF_{(-)})}(-)$ is less well-behaved. 

    Bouc and Yal\c{c}in's short exact sequence above determines the kernel of the Bouc homomorphism $\Psi: CF(G) \to D^\Omega(G)$ when $G$ is a $p$-group. The Bouc homomorphism was generalized to all finite groups in \cite[Theorem 1.4]{GeYa21}, and in general, we have an inclusion $\ker(\Psi)\subseteq CF_b(G,p)$. Gelvin and Yal\c{c}in in \cite{GeYa21} proposed an additional condition, the \textit{oriented Artin condition}, and proved that $CF_{ba^+}(G,p) \subseteq \ker(\Psi) \subseteq CF_b(G,p)$, where $CF_{ba^+}(G,p)$ denotes the subgroup of $CF(G,p)$ consisting of superclass functions satisfying the oriented Artin-Borel-Smith conditions (see Definition \ref{def:orientedartin}). Furthermore, they posed the question of whether in general, $CF_{ba^+}(G,p) = \ker(\Psi)$. We provide a positive answer to this question by rephrasing it into a question regarding a particular class of endotrivial complexes.

    \begin{theorem*}[Theorem \ref{thm:bouchomker}]
        We have $\ker(\Psi) = CF_{ba^+}(G,p)$.
    \end{theorem*}

    The proof relies on the Lefschetz homomorphism $\Lambda: \calE_k(G) \to O(T(kG))$ induced by taking the Lefschetz invariant (also called the Euler characteristic) of a chain complex (i.e. an alternating sum of its terms), and an equivalent formulation of the trivial source ring $T(kG)$, the Grothendieck group of $K^b({}_{kG}\triv)$ as a tensor-triangulated category, developed by Boltje and Carman in \cite{BC23}. In \cite{SKM23}, we determined that the Lefschetz homomorphism is not surjective in general for non-$p$-groups. In particular, the image of $\Lambda$ satisfies a certain Galois invariance condition. However, the question remained open for $p$-groups. Using a classical result of Tornehave, we determine that indeed, $\Lambda$ is surjective for any $p$-group. 

    \begin{theorem*}[Theorem \ref{thm:surjetivityoflefschetz}]
        Let $G$ be a $p$-group. $\Lambda: \calE_k(G) \to O(T(kG))$ is surjective.
    \end{theorem*}

    We obtain as a rather striking and surprising corollary that if $G$ is a $p$-group, every $p$-permutation autoequivalence of the block algebra $kG$ is induced by a splendid Rickard autoequivalence of $kG$. 

    \begin{theorem*}[Theorem \ref{thm:ppermsliftforpgroups}]
        Let $G$ be a $p$-group and let $\gamma \in O(T^\Delta(kG,kG))$ be a $p$-permutation autoequivalence of $kG$. There exists a splendid Rickard equivalence $\Gamma$ inducing $\gamma$, i.e. we have $\Lambda(\Gamma)=\gamma.$
    \end{theorem*}

    Finally, we remark that the classification of endotrivial complexes should have significant implications for our understanding of the geometry of the tensor-triangulated category $K^b({}_{kG}\triv)$. Moreover, by the equivalences of categories outlined in \cite{BG23b}, Corollary \ref{thm:borelsmithfornonpgroups} also provides the classification of invertible objects in certain categories of Artin motives and cohomological Mackey functors. We describe one such possibility the classification opens: in \cite[Section 3]{BG23}, a ``twisted cohomology ring'' is constructed to determine the topology of the Balmer spectrum of $K^b({}_{kG}\triv)$. This construction seems to implicitly use the classification of endotrivial complexes for elementary abelian $p$-groups, therefore it is reasonable to expect that a more general twisted cohomology ring for any group, corresponding with the classification of endotrivial complexes for that group, should exist. This also opens the door for gluing techniques with respect to localization, akin to those performed in \cite{BBC09}. However, we leave this exploration for a future article - the scope of this paper does not include tensor-triangulated geometry. 

    \textbf{Notation and conventions:} For the paper, we fix $p$ a prime, $k$ a field of characteristic $p$, and $G$ a finite group. ${}_{kG}\triv$ denotes the category of finitely generated $p$-permutation $kG$-modules, and $K^b({}_{kG}\triv)$ denotes its bounded homotopy category. We denote the Brauer construction (also called the Brauer quotient) at a $p$-subgroup $P$ of $G$ by $-(P): {}_{kG}\triv \to {}_{k[N_G(P)/P]}\triv$. This extends to a functor $-(P): K^b({}_{kG}\triv) \to K^b({}_{k[N_G(P)/P]}\triv)$. We write $s_p(G)$ for the set of $p$-subgroups of $G$, $[s_p(G)]$ for a set of representatives of conjugacy classes of $s_p(G)$, and $\Syl_p(G)$ for the set of Sylow $p$-subgroups of $G$. We let $S \in \Syl_p(G)$ denote an arbitrary Sylow $p$-subgroup of $G$. Given a $kG$-module $M$, we let $M[i]$ denote the chain complex with $M$ in degree $i$ and the zero module in all other degrees. We denote the additive group of all $\Z$-valued superclass functions of $G$, $f: \text{sub}(G) \to \Z$, by $CF(G)$ and the additive group of all $\Z$-valued superclass functions on the set of $p$-subgroups of $G$, $f: s_p(G) \to \Z$, by $CF(G,p)$.

    \section{$V$-endosplit-trivial complexes}

    We refer the reader to \cite{CL11} for review of projectivity relative to modules, absolute $p$-divisibility, and relatively endotrivial modules, to \cite{SKM23} for review of $p$-permutation modules. the Brauer construction, and endotrivial modules. Additionally, we refer the reader to \cite{Bou10} for a complete overview of biset functors, and to \cite[Section 2]{BoYa06} for a brief overview of rational $p$-biset functors.

    Let $V$ be an absolutely $p$-divisible $kG$-module, possibly 0. That is, every indecomposable direct summand of $V$ has $k$-dimension divisible by $p$. We first review some results regarding $V$-endosplit-trivial chain complexes and endotrivial chain complexes.

    \begin{definition}
        Let $C \in K^b({}_{kG}\triv)$. Say $C$ is a \textit{$V$-endosplit-trivial complex} if \[\End_k(C) \cong C^* \otimes_k C \simeq (k \oplus M)[0],\] where $M$ is a $V$-projective $kG$-module.

        If $V = 0$, we say $C$ is an \textit{endotrivial complex} for short. This coincides with the definition of an endotrivial complex given in \cite{SKM23}, that is, we have $C^* \otimes_k C \simeq k[0].$

    \end{definition}

    If $V$ is an absolutely $p$-divisible $kG$-module (possibly $0$), we define $\calX_V \subseteq s_p(G)$ to be the set of $p$-subgroups of $G$ for which $V(P) = 0$. 

    \begin{prop}{(Omnibus properties)}
        Let $C$ be a $V$-endosplit-trivial complex. Then the following hold:
        \begin{enumerate}
            \item $C$ has a unique indecomposable summand $C_0$ which is $V$-endosplit-trivial, and all other direct summands are $V$-projective or contractible. We say $C_0$ is the \textit{cap} of $C$.
            \item There exists a unique $i \in \Z$ for which $H_i(C) \neq 0$.
            \item $C_0$ has vertex set $\Syl_p(G)$.
            \item If $H \leq G$, then $\Res^G_H C$ is a $\Res^G_H V$-endosplit-trivial complex.
            \item If $G$ is a quotient of $\Tilde{G}$, then $\Inf^{\Tilde{G}}_G C$ is a $\Inf^{\Tilde{G}}_G V$-endosplit-trivial complex.
            \item If $P \in s_p(G)$, then the chain complex of $k[N_G(P)/P]$-modules $C(P)$ is a $V(P)$-endosplit-trivial complex. In particular, if $P \in \calX_V$, i.e. $V(P) = 0$, then $C(P)$ is an endotrivial complex.
            \item If $H_i(C) \neq 0$, then $H_i(C)$ is a relatively $V$-endotrivial $kG$-module.
        \end{enumerate}
    \end{prop}
    \begin{proof}
        Proofs of all these statements are fairly straightforward and can be found in \cite{SKM24a}.
    \end{proof}

    \begin{theorem}{\cite[Corollary 8.4]{SKM24a}}\label{thm:equivdefendosplittriv}

        Let $C \in Ch^b({}_{kG}\triv)$. The following are equivalent:
        \begin{enumerate}
            \item $C$ is $V$-endosplit-trivial.

            \item For all $P \in s_p(G)$, $C(P)$ has nonzero homology concentrated in exactly one degree, and if $P\in \calX_V$, that homology has $k$-dimension one.

            \item For all $P \in s_p(G)$, $C(P)$ has nonzero homology concentrated in exactly one degree $i$ and $H_i(C)$ is a $V$-endotrivial $kG$-module.

            \item $C$ is isomorphic to a shift of an endosplit $p$-permutation resolution of a $V$-endotrivial $kG$-module.
        \end{enumerate}

        In particular, $C$ is endotrivial if and only if for all $P \in s_p(G)$, $C(P)$ has nonzero homology concentrated in one degree with that homology having $k$-dimension one.

    \end{theorem}

    \begin{definition}\label{def:grpofetrivs}
        We define the group $\calE^V_k(G)$ of $V$-endosplit-trivial chain complexes as follows. We say two $V$-endosplit-trivial complexes $C_1, C_2$ are equivalent, written $C_1 \sim C_2$, if and only if $C_1$ and $C_2$ have isomorphic caps. Set $\calE_k^V(G)$ to be the set of equivalence classes of $V$-endosplit-trivial complexes, with group addition induced from $\otimes_k$. 
        
        If $C$ is a $V$-endosplit-trivial complex, we write $[C]\in \calE^V_k(G)$ to denote the corresponding class of complexes in the group. We set $\calE_k(G) := \calE^{\{0\}}_k(G)$. This construction agrees with the definition of $\calE_k(G)$ given in \cite{SKM23}, and is the Picard group of $K^b({}_{kG}\triv).$

    \end{definition}

    \begin{definition}
        \begin{enumerate}
            \item \cite[Proposition 3.5.1]{CL11} Let $V$ be an absolutely $p$-divisible $kG$-module. Define an equivalence relation $\sim_V$ on the class of $V$-endotrivial $kG$-modules as follows: given two $V$-endotrivial modules $M,N$, write $M\sim_V N$ if and only if $M$ and $N$ have isomorphic caps, or equivalently, if $M \cong N$ in ${}_{kG}\underline\catmod_V$ (see \cite[Section 4]{SKM24a}). Let $T_V(G)$ denote the resulting set of equivalence classes, then $\otimes_k$ induces an abelian group structure on $T_V(G)$.
            \item Let $S\in \Syl_p(G)$. We define $T_V(G,S) \leq T_V(G)$ as the kernel of $\Res^G_S: T_V(G) \to T_{\Res^G_S V}(S)$. Equivalently, $T_V(G,S)$ is the subgroup of $T_V(G)$ generated by trivial source $V$-endotrivial $kG$-modules; see \cite[Theorem 5.5]{SKM24a} for details. Such modules are also called \textit{Sylow-trivial}. 
        \end{enumerate}
    \end{definition}

    \begin{remark}\label{rem:identification}
        Observe that we have injective group homomorphisms $\calE_k(G) \to \calE_k^V(G)$ given by the obvious inclusion and $T_V(G,S) \to \calE_k^V(G)$ given by $[M] \mapsto [M[0]]$. In this way, we identify both $T_V(G,S)$ and $\calE_k(G)$ as subgroups of $\calE_k^V(G)$. It is straightforward to see the image of the inclusion $\calE_k(G) \hookrightarrow \calE_k^V(G)$ consists of all equivalences classes of $V$-endosplit-trivial complexes $[C]$ for which $C$ has a cap with $k$-dimension one. 

        Recall $CF(G,p)$ denotes the set of $\Z$-valued superclass functions on the set of $p$-subgroups of $G$. When $G$ is a $p$-group, $CF(G) = B(G)^*$, the $\Z$-dual of the Burnside ring $B(G)$, via the identification \[CF(G) \to B(G)^*, f \mapsto \left([G/H] \mapsto f(H)\right).\] In this situation, these groups arise from isomorphic rational $p$-biset functors, which we describe in Section \ref{section:caseofpgroups}.
    \end{remark}

    \begin{definition}
        Let $C$ be a $V$-endosplit-trivial chain complex. The \textit{h-mark of $C$ at $P$}, denoted $h_C(P)$, is the unique integer for which $H_{h_C(P)}(C(P)) \neq 0$. The \textit{homology of $C$ at $P$}, denoted $\calH_C(P)$, is $[H_{h_C(P)}(C(P))] \in T_{V(P)}(N_G(P)/P)$. If $V = 0$, we identify $T_0(G) = X(G)$, the group of isomorphism classes of $k$-dimension one $kG$-modules, with group law induced by $\otimes_k$. $X(G)$ identifies with the group of group homomorphism $\Hom(G,k^\times)$ by sending a group homomorphism $\omega: G\to k^\times$ to the $k$-dimension one $kG$-module $k_\omega$ defined by $g\cdot m := \omega(g)m$ for $m \in k_\omega, g\in G$.

        Observe $h_C$ is a superclass function $h_C: s_p(G) \to \Z$, since for any $g \in G$ and $P \in s_p(G)$ $C({}^gP)\cong ({}^g C)(P),$ see \cite[Proposition 2.7(c)]{SKM23}.
    \end{definition}

    \begin{prop}{\cite[Definition 3.5]{SKM23} and \cite[Theorem 9.7]{SKM24a}}\label{hmarkhoms}
        We have a well-defined, injective group homomorphism:

        \begin{align*}
            \Xi: \calE_k^V(G) &\to \prod_{P \in s_p(G)}\left( \Z \times T_{V(P)}(N_G(P)/P) \right)\\
            [C] &\mapsto (h_C(P), \calH_C(P))_{P \in s_p(G)}.
        \end{align*}

        Regarding $h_C$ as a $\Z$-valued superclass function on $s_p(G)$ gives another group homomorphism, which we call the \textit{h-mark homomorphism},

        \begin{align*}
            h: \calE_k^V(G) &\to CF(G,p)\\
            [C] &\mapsto h_C
        \end{align*}

        We have $\ker(h) \cong T_V(G,S)$, the torsion subgroup of $\calE_k^V(G)$. In particular, $\calE_k^V(G)$ is a finitely generated abelian group.
    \end{prop}

    \subsection{The module $V(\calF_G)$}

    We next introduce a specific absolutely $p$-divisible $kG$-module which plays a crucial role in the sequel. This module was used by Lassueur in \cite{CL13} in her construction of the generalized Dade group.

    \begin{definition}
        We define the $kG$-module $V(\calF_G)$ as follows: \[V(\calF_G) = \bigoplus_{Q \in s_p(G)\setminus \Syl_p(G)} k[G/Q].\] It easily follows from definition that a $kG$-module $M$ is $V(\calF_G)$-projective if and only if every indecomposable summand of $M$ has non-Sylow vertices. 
    \end{definition}

    $V(\calF_G)$ is absolutely $p$-divisible, and in a sense is the ``largest'' absolutely $p$-divisible module via the next theorem. 
    \begin{theorem}\label{thm:universalabsolutepdiv}
        Let $V$ be an absolutely $p$-divisible permutation $kG$-module.
        \begin{enumerate}
            \item If $M$ is a $V$-projective $kG$-module, $M$ is $V(\calF_G)$-projective.
            \item If $C$ is a $V$-endosplit-trivial $kG$-complex, $C$ is $V(\calF_G)$-endosplit-trivial.
        \end{enumerate}

    \end{theorem}
    \begin{proof}
        Let $H \leq G$; we denote the full subcategory of $H$-projective $kG$-modules by ${}_{kG}Proj(H)$. Similarly, for a $kG$-module $V$, we denote the full subcategory of $V$-projective $kG$-modules by ${}_{kG}Proj(V)$. Recall that that ${}_{kG}Proj(H) = {}_{kG}Proj(k[G/H])$. For (a), write $V = V_1 \oplus\cdots \oplus V_n$ with each $V_i$ a transitive permutation module. Since $V$ is absolutely $p$-divisible, each $V_i$ has vertex strictly contained in a Sylow $p$-subgroup. Then we have ${}_{kG}Proj(V_i) = {}_{kG}Proj(Q_i)$ with $Q_i$ the stabilizer of $V_i$. We have \[{}_{kG}Proj(V) = \bigoplus_{i=1}^n {}_{kG}Proj(V_i) \subseteq \bigoplus_{i=1}^n {}_{kG}Proj(Q_i) \subseteq {}_{kG}Proj(V(\calF_G)). \] Now (b) follows directly from (a) and Theorem \ref{thm:equivdefendosplittriv}.
    \end{proof}

    \begin{definition}
        Given any $G$-set $X$, we write $\Delta(X)$ for the $kG$-module given by the kernel of the augmentation homomorphism, \[kX \to k,\, x \mapsto 1.\] This is referred to as the \textit{relative syzygy of $X$}.
    \end{definition}

    \begin{theorem}\label{thm:basisofcalEVFG}
        Let $V = V(\calF_G)$. Then $h: \calE_k^V(G) \to CF(G,p)$ is surjective. In particular, if $G$ is a $p$-group, we have an isomorphism $\calE_k^V(G) \cong CF(G) \cong B(G)^*$.
    \end{theorem}
    \begin{proof}
        We give an explicit set of $V$-endosplit-trivial chain complexes $C_Q$ for every $Q \in [s_p(G)]$ such that the set $\{h_{C_Q}\}$ forms a basis of $CF(G,p)$.

        If $S \in \Syl_p(G)$, we define $C_S = k[1]$. Otherwise, if $Q \not\in\Syl_p(G)$, we define \[C_Q = 0\to k[G/Q] \to k \to 0,\] with $k$ in degree 0, and the differential given by the augmentation map, that is, the homomorphism induced by $gQ \mapsto 1 \in k$. $C_Q$ is a $V$-endosplit-trivial chain complex, as the augmentation map $k[G/Q] \to k$ is $Q$-split, hence $k[G/Q]$-split (see \cite{Al01}). Therefore $H_1(C_Q) \cong \Delta(G/Q)$ is $k[G/Q]$-endotrivial, hence $V(\calF_G)$-endotrivial. It follows that \[h_{C_Q}(P) = \begin{cases}1 & P \leq_G Q\\ 0 & P \not\leq_G Q\end{cases}.\] It follows by a standard M\"obius inversion argument on the poset of $p$-subgroups of $G$ that the set $\{h_{C_Q}\}_{Q \in [s_p(G)]}$ forms a $\Z$-basis of $CF(G,p)$, as desired. The last statement follows immediately, since if $G$ is a $p$-group, $h$ is injective as well, by Proposition \ref{hmarkhoms} and the fact that $T_V(G,S) = T_V(G, G) = \{k[0]\}$.
    \end{proof}

    \begin{remark}
        Explicitly, if $\{e_Q\}_{Q\in [s_p(G)]}$ is the idempotent basis for $CF(G,p)$, i.e. $e_Q(P) = 1$ if $Q =_G P$ and $e_Q(P) = 0$ otherwise, then \[e_Q = \sum_{P \in [s_p(G)]} \mu(P,Q)h_{C_P},\] where $\mu$ denotes the M\"obius function for the poset $[s_p(G)]$.
    \end{remark}

    \begin{definition}\label{basisnotationdefinitions}
        For the rest of this paper, we will always use the notation used in the previous proof for the $V(\calF_G)$-endosplit-trivial complex $C_Q$. We write $\omega_Q$ for $h_{C_Q}$. We will use extensively the facts that $\{C_Q\}_{Q\in [s_p(G)]}$ forms a $\Z$-basis of a free complement of $T_{V(\calF_G)}(G,S) \leq \calE_k^{V(\calF_G)}(G)$ (the torsion subgroup of $ \calE_k^{V(\calF_G)}(G)$), and $\{\omega_Q\}_{Q \in [s_p(G)]}$ forms a $\Z$-basis of $CF(G,p)$.

        Moreover, observe that when $Q\not\in \Syl_p(G)$, $\calH_{C_Q}(1) = [\Delta(G/Q)] \in T_{V(\calF_G)}(G)$, and for $S \in \Syl_p(G)$, $H_{C_S}(1) = [k]$.

        Extending notation, for a $G$-set $X$, define $\omega_X \in CF(G,p)$ as follows: \[\omega_X(P) = \begin{cases} 1 & X^P \neq \emptyset \\ 0 & \text{otherwise} \end{cases}.\] In this case, $\omega_{G/Q} = \omega_Q$.
    \end{definition}

    Using Theorem \ref{thm:basisofcalEVFG}, we obtain a construction for a $V$-endosplit-trivial chain complex $C$ given only its h-marks, up to a twist by a unique trivial source (i.e. indecomposable $p$-permutation) $V$-endotrivial $kG$-module. It suffices to assume $V$ is permutation by \cite[Proposition 3.10]{SKM24a}.

    \begin{theorem}\label{thm:hmarktocomplex}
        Let $V$ be an absolutely $p$-divisible permutation $kG$-module (possibly 0), and $C$ be an indecomposable $V$-endosplit-trivial chain complex of $kG$-modules with h-marks $h_C \in CF(G,p)$. For $P \in [s_p(G)]$, set \[b_P:= \sum_{Q \in [s_p(G)]} h_{C}(Q)\mu(P,Q).\]

        There exists a unique trivial source $V(\calF_G)$-endotrivial $kG$-module $M$ for which the following holds: $C$ is isomorphic to the unique $V$-endosplit-trivial indecomposable direct summand of the $V(\calF_G)$-endosplit-trivial chain complex
        \[M\otimes_k \bigotimes_{Q \in [s_p(G)]} C_Q^{\otimes b_Q},\]where $C_Q$ is defined in the proof of Theorem \ref{thm:basisofcalEVFG}.

    \end{theorem}
    \begin{proof}
        Let $h_C$ denote the corresponding h-marks of $C$. Note $C$ is a $V(\calF_G)$-endosplit-trivial complex as well by Theorem \ref{thm:universalabsolutepdiv}, since $V$ is absolutely $p$-divisible. Then we have
        \begin{align*}
            h_C = \sum_{Q \in [s_p(G)]}a_Q e_Q &= \sum_{Q \in [s_p(G)]} a_Q \left(\sum_{P\in s_p(G)} \mu(P,Q)\omega_P\right) \\
            &= \sum_{P \in [s_p(G)]} \omega_P \left(\sum_{Q \in [s_p(G)]} a_Q \mu(P,Q) \right)= \sum_{P \in [s_p(G)]} b_P\omega_P.
        \end{align*}
        It follows that the $V(\calF_G)$-endosplit-trivial chain complex \[D := \bigotimes_{Q \in [s_p(G)]} C_Q^{\otimes b_Q}\] has the same h-marks as $C$, and thus $D^* \otimes_k C$ has h-marks all zero. By Proposition \ref{hmarkhoms}, \[[M[0]] + [D] = [C] \in \calE_k^{V(\calF_G)}(G),\] for some class of $V(\calF_G)$-endotrivial $p$-permutation modules $[M[0]]$. Since every class of relatively $V(\calF_G)$-endotrivial modules has a unique indecomposable representative, $M$ can be chosen to be indecomposable. Since every class of $V(\calF_G)$-endosplit-trivial complexes has a unique summand with Sylow vertices, and this chain complex is also $V(\calF_G)$-endosplit-trivial, $C$ is the unique indecomposable $V(\calF_G)$-endosplit-trivial summand of $M[0]\otimes_k D$, as desired. 
    \end{proof}

    \begin{remark}
        In particular, if $G$ is a $p$-group, $C$ is isomorphic to the unique indecomposable $V(\calF_G)$-endosplit-trivial direct summand of \[\bigotimes_{Q \in [s_p(G)]} C_Q^{\otimes b_Q},\] since in this case, the only trivial source relatively endotrivial $kG$-module is $k$. One can reconstruct an endotrivial $kG$-complex purely from knowing its h-marks, but explicitly determining such complexes remains difficult. 
    \end{remark}

    \begin{remark}
        We briefly note on endosplit-trivial complexes in the wild. In this paper, we are mostly concerned with either endotrivial complexes (i.e. $\{0\}$-endosplit-trivial complexes) or $V(\calF_G)$-endosplit-trivial complexes. Another choice of module $V$ which produces a class of $V$-endosplit-trivial complexes which is large enough to be interesting but small enough to be understood is $V = k[G/1]$; in this case a $k[G/1]$-endosplit-trivial complex is a bounded chain complex $C$ of $p$-permutation $kG$-modules for which $C^* \otimes_k C \simeq (k \oplus M)[0]$ for a projective $kG$-module $M$. In particular, the homology of $C$ is necessarily an endotrivial module (though not necessarily Sylow-trivial). Certain $k[G/1]$-endosplit-trivial complexes (more accurately, ``twisted'' Steinberg complexes over the orbit category $\calO_p^*(G)$, see \cite[3.2.2]{Gr23}) play a fundamental role in Grodal's main theorem \cite[Theorem A]{Gr23}: there exists an isomorphism of abelian groups $T_k(G,S) \cong H^1(\calO_p^*(G), k^\times)$.
        
        In \cite{SKM24a}, multiple weaker notions of relatively endotrivial complexes are defined: one may alternatively consider bounded chain complexes $C$ of $p$-permutation $kG$-modules for which either $C^* \otimes_k C \simeq k[0] \oplus D$ for $D$ a chain complex of $V$-projective $kG$-modules, or $D$ a $V$-projective chain complex; in these cases, $C$ is respectively a weakly $V$-endotrivial complex or a strongly $V$-endotrivial complex. These versions of relative endotriviality each have corresponding parametrizing subgroups. For $V = k[G/1]$, \cite[Theorems 11.1, 11.4]{SKM24a} assert such groups are isomorphic modulo truncated projective resolutions or injective hulls; for general absolutely $p$-divisible $kG$-modules $V$, it is unknown if an analogue holds.  

    \end{remark}

    \section{Capped endosplit $p$-permutation resolutions and the Dade group of a finite group}

    In this section, we construct a short exact sequence relating $V$-endosplit-trivial complexes, endotrivial complexes, and the generalized Dade group as defined by Lassueur in \cite{CL13}. We also determine up to equivalence which strongly capped endo-$p$-permutation modules, as defined by Lassueur, have endosplit $p$-permutation resolutions.

    \begin{definition}\label{dadegroupgeneralizations}
        \begin{enumerate}
            \item We say a $kG$-module $M$ is \textit{endopermutation} (resp. \textit{endo-$p$-permutation}) if $M^* \otimes_k M$ is a permutation module (resp. $p$-permutation module).
            \item The following definition is due to Dade in \cite{Da78}. Let $P$ be a $p$-group. We say an endopermutation $kP$-module $M$ is \textit{capped} if $M$ has a direct summand with vertex $P$. Equivalently, $M$ is capped if $(M^* \otimes_k M)(P)$ is nonzero, and equivalently, $M^* \otimes_k M$ has $k$ as a direct summand. If $M_0$ is an indecomposable summand of $M$ with vertex $P$, we say $M_0$ is the \textit{cap} of $M$.

            Define an equivalence relation on the set of capped endopermutation $kP$-modules as follows: say $M \sim N$ if and only if $M \oplus N$ is an endopermutation $kP$-module. Equivalently, $M\otimes_k N^*$ is endopermutation. It follows by a lemma of Dade that each equivalence class has a unique indecomposable representative. Therefore, two endopermutation $kP$-module are equivalent if and only if they have isomorphic caps. Write $D_k(P)$ for the set of equivalence classes of capped endopermutation modules. This forms an abelian group with addition induced by $\otimes_k$, and is called \textit{the Dade group}.

            \item A generalization of the Dade group to non-$p$-groups was first constructed by Linckelmann and Mazza in \cite{LiMa09}, and later described using separate techniques by Lassueur in \cite{CL13}. We give the definition in \cite{CL13}. An endo-$p$-permutation $kG$-module $M$ is \textit{strongly capped} if $M$ is $V(\calF_G)$-endotrivial. In this case, $M$ has a unique indecomposable strongly capped direct summand, and all other direct summands are $V(\calF_G)$-projective. As before, say this unique summand is the \textit{cap} of $M$.

            Define an equivalence relation on the set of strongly capped endo-$p$-permutation modules as before: given two endo-$p$-permutation $kG$-modules $M,N$, say $M \sim N$ if and only if $M$ and $N$ have isomorphic caps. It follows that each equivalence class has a unique indecomposable representative. Write $D_k(G)$ for the set of equivalence classes of strongly capped endo-$p$-permutation modules. This forms an abelian group as before with addition induced by $\otimes_k$, and is called \textit{the generalized Dade group}. If $G$ is a $p$-group, we recover the classical Dade group, and for this reason, we refer to $D_k(G)$ as the \textit{Dade group} without further mention.

            Recall that $X(G)$ denotes the set of isomorphism classes of $k$-dimension one $kG$-modules. We have a series of inclusions \[X(G) \hookrightarrow T_{V(\calF_G)}(G,S) \hookrightarrow D_k(G) \hookrightarrow T_{V(\calF_G)}(G).\] In general, these maps will not be surjective.

            \item We extend this notion to $V$-endosplit-trivial $kG$-complexes. Say a chain complex of $kG$-modules $C$ is \textit{capped} if $C$ is a $V(\calF_G)$-endosplit-trivial complex. In this case, $C$ is a shifted endosplit $p$-permutation resolution of a strongly capped endo-$p$-permutation $kG$-module (see \cite[Remark 9.11]{SKM24a}).

            The equivalence relation used for the prior definition of $D_k(G)$ is already encoded in $\calE_k^{V(\calF_G)}(G)$. In this way, $\calE_k^{V(\calF_G)}(G)$ is a chain-complex theoretic ``Dade group.''
        \end{enumerate}

    \end{definition}

    \begin{definition}
        \begin{enumerate}
            \item For a $G$-set $X$, recall $\Delta(X)$ is the kernel of the augmentation homomorphism $kX \twoheadrightarrow k$. Define the element $\Omega_X \in D_k(G)$ as follows:
            \[\Omega_X = \begin{cases}[\Delta(X)] & \text{if } X^S = \emptyset. \\ [k] & \text{if } X^S \neq \emptyset.\end{cases}\] This is well-defined since $\Delta(X)$ is $V(\calF_G)$-endotrivial. Indeed, the chain complex $kX \twoheadrightarrow k$ is a $V(\calF_G)$-endosplit-trivial complex.
            If $X = G/P$ for some $P \in s_p(G)$, we set $\Omega_P := \Omega_{G/P}$.
            \item Let $D^\Omega_k(G) \leq D_k(G)$ be the subgroup generated by elements of the form $\Omega_X$ where $X$ runs over all $G$-sets. When $G$ is a $p$-group, this group is independent of choice of field $k$, so we write $D^\Omega(G)$ (see for instance \cite[Theorem 12.9.10]{Bou10}).
        \end{enumerate}
    \end{definition}

    The following theorem gives a generating set of $D_k^\Omega(G)$.

    \begin{prop}{\cite[Lemma 12.1]{CL13}}\label{gensetofDomega}
        The group $D^\Omega_k(G)$ is generated by the relative syzygies $\Omega_{G/Q}$, where $Q$ runs over the proper subgroups of $S$.
    \end{prop}

    This was proven using separate techniques in \cite[Proposition 5.14]{GeYa21} as well.

    \begin{remark}\label{relsyzygies}
        Let $S \in \Syl_p(G)$. For any $G$-set $X$ for which $X^S = 0$, $\Delta(X)$ is endo-$p$-permutation and has an endosplit $p$-permutation resolution (if $X^S \neq 0$, $\Delta(X)$ is permutation). Indeed, the chain complex $kX \twoheadrightarrow k$ suffices.

        It follows that every element of $D^\Omega_k(G)$ has a representative which has an endosplit $p$-permutation resolution. Indeed, if \[[M] = \sum_{i=1}^j \Omega_{Q_i}^{\epsilon_i} \in D_k^\Omega(G),\] with $\epsilon_i \in\{\pm 1\}$ and $Q_i$ a $p$-subgroup of $G$ which is non-Sylow, then a corresponding endosplit $p$-permutation resolution of $[M]$ (up to some representative) is \[\bigotimes_{i=1}^j [C_Q]^{\epsilon_i} \in \calE_k^{V\calF(G)}(G).\] In particular, given any $[M] \in D^\Omega_k(G)$, there exists a $V(\calF_G)$-endosplit-trivial complex $C$ with $\calH_C(1) = [M] \in D_k(G)$.

        However, every element which has an endosplit $p$-permutation resolution does not necessarily belong to an equivalence class in $D^\Omega_k(G)$. For instance, if $G$ is not a periodic group, any class in $D_k(G)$ with a representative given by a nontrivial $k$-dimension one representation has an endosplit $p$-permutation resolution but is not an element of $D^\Omega_k(G)$, see \cite[Remark 12.2]{CL13}. On the other hand, it is shown in \cite[Section 12]{CL13} that in many cases, $D^\Omega_k(G) + \Gamma(X(N_G(S))) = D_k(G)$, where $\Gamma(X(N_G(S)))$ is the subgroup of $D_k(G)$ generated by Green correspondents of 1-dimensional $k[N_G(S)]$-modules. By \cite[Proposition 4.1]{CL13}, $\Gamma(X(N_G(S)))$ is a well-defined subgroup of $D_k(G)$, that is, each representative of each equivalence class in $\Gamma(X(N_G(S)))$ is $V(\calF_G)$-endotrivial.

        This observation is an analogue of the characterization for $p$-groups that $D_k(G) = D^\Omega(G)$ for any $p$-group which does not contain a generalized quaternion group in its genetic bases (see \cite[Theorem 12.10.4]{Bou10}). It is unknown when this equality holds in general for non-$p$-groups. Note that the elements of $\Gamma(X(N_G(S)))$ have trivial source representatives, so every element of the subgroup $D^\Omega_k(G) + \Gamma(X(N_G(S)))$ has an endosplit $p$-permutation resolution.

        Finally, the sum $D^\Omega_k(G) + \Gamma(X(N_G(S)))$ may not be direct. See the comment after \cite[Theorem 12.6]{CL13}.
    \end{remark}

    In fact, $\Gamma(X(N_G(S))) = T_{V(\calF_G)}(G,S)$, by \cite[Proposition 4.1(d)]{CL13}. From this point on, we replace $\Gamma(X(N_G(S)))$ with $T_{V(\calF_G)}(G,S)$.

    \begin{definition}
        Denote by $\calT\calE_k(G)$ the subgroup of $\calE_k(G)$ consisting of endotrivial complexes $C$ for which $\calH_C(1) = [k]$. We call this subgroup \textit{the homology-normalized subgroup of $\calE_k(G)$}. $\calT\calE_k(G)$ identifies with the subgroup of $\calE_k^{V(\calF_G)}(G)$ consisting of equivalence classes of complexes $C$ for which $\calH_C(1) = [k]$ as well.
    \end{definition}

    Since $V(\calF_G)$-endosplit-trivial complexes are equivalently shifted endosplit $p$-permutation resolutions of $V(\calF_G)$-endotrivial $kG$-modules, the image of $\calH$ consists of all equivalence classes of strongly capped endo-$p$-permutation modules which have endosplit $p$-permutation resolutions. The next theorem characterizes precisely which strongly capped endo-$p$-permutation modules, up to equivalence in the Dade group, have an endosplit $p$-permutation resolution.

    \begin{theorem}\label{thm:whichmoduleshaveepprs}
        Let $M$ be a strongly capped endo-$p$-permutation $kG$-module. $[M] \in D_k(G)$ contains a representative which has an endosplit $p$-permutation resolution if and only if $[M] \in D_k^\Omega(G) + T_{V(\calF_G)}(G,S)$.

    \end{theorem}
    \begin{proof}
        The reverse direction follows since any $\Omega_X \in D_k^\Omega(G)$ has the corresponding endosplit $p$-permutation resolution $kX \twoheadrightarrow k$ and any $[M] \in T_{V(\calF_G)}(G,S)$ has $M[0]$ as an endosplit $p$-permutation resolution.

        For the forward direction, suppose $M$ has an endo-$p$-permutation resolution. Since $M$ is $V(\calF_G)$-endotrivial, there exists an $V(\calF_G)$-endosplit-trivial complex $C$ with $\calH(C) = [M] \in D_k(G)$.

        Theorem \ref{thm:basisofcalEVFG} implies that $C$ shares a cap with a tensor product of the chain complexes $C_{Q}$, their duals, and elements of $T_{V(\calF_G)}(G,S)$ considered as chain complexes in degree 0. It follows that \[[M] = \calH(C) = [M_0] +\sum_{i=1}^j \Omega_{Q_i}^{\pm 1} \in D_k(G),\] where each $Q_i$ is some $p$-subgroup which is non-Sylow, and $M_0$ is a trivial source $V(\calF_G)$-endotrivial module. Thus $[M] \in D_k^\Omega(G) + T_{V(\calF_G)}(G,S)$, as desired.
    \end{proof}

    \begin{theorem}\label{shortexactsequence}\label{kernelofhomology}
        We have a well-defined group homomorphism $\calH: \calE_k^{V(\calF_G)}(G) \to D_k(G)$ induced by $[C] \mapsto \calH_C(1)$. $\ker(\calH) = \calT\calE_k(G)$ and $\im (\calH) = D^\Omega_k(G) + T_{V(\calF_G)}(G,S)$, so we have a short exact sequence of abelian groups \[0\to \calT\calE_k(G) \to \calE_k^{V(\calF_G)}(G) \xrightarrow[]{\calH} D_k^\Omega(G) + T_{V(\calF_G)}(G,S) \to 0.\]

        In particular, if $G$ satisfies $D^\Omega_k(G) + T_{V(\calF_G)}(G,S) = D_k(G)$, then $\calH$ is surjective.
    \end{theorem}
    \begin{proof}
        $\calH$ is a well-defined group homomorphism from Theorem \ref{hmarkhoms} after projection, and Theorem \ref{thm:basisofcalEVFG}, which demonstrates that the image of $\calH$ is generated by endo-$p$-permutation $kG$-modules. In particular, $\calE_k^{V(\calF_G)}$ is spanned by $\{C_Q\}_{Q \in [s_p(G)]}$ and $T_V(G,S)$, all of which have endo-$p$-permutation homology.

        Given any $V(\calF_G)$-endosplit-trivial complex $C$, $\calH(C) = [k] \in D_k(G)$ if and only if $\calH(C) = [k] \in T_{V(\calF_G)}(G)$. From this, it follows by definition that $\ker(\calH) = \calT\calE_k(G)$.

        Now, $\im (\calH) = D^\Omega_k(G) + T_{V(\calF_G)}(G,S)$ follows directly from Theorem \ref{thm:whichmoduleshaveepprs}, since $V(\calF_G)$-endosplit-trivial complexes are equivalently shifted endosplit $p$-permutation resolutions. The existence of the short exact sequence follows immediately.
    \end{proof}

    \begin{remark}
        In particular, $\calH: \calE_k^{V(\calF_G)}(G) \to D_k(G)$ is surjective in the following cases:
        \begin{itemize}
            \item $G$ is a $p$-group which is not generalized quaternion.
            \item $G$ has a cyclic Sylow $p$-subgroup.
            \item $p$ is odd and $G$ has a normal Sylow $p$-subgroup.
            \item $N_G(S)$ controls $p$-fusion of $G$.
            \item $G = GL_3(\mathbb{F}_p)$ for odd $p$.
        \end{itemize}
        Surjectivity when $G$ is a $p$-group which is not generalized quaternion follows by the characterization of the Dade group (see \cite[Chapter 12]{Bou10} or \cite{CaTh00}), as $D_k(G)$ is generated by relative syzygies. The other cases are known examples of when $D_k^\Omega(G) + T_{V(\calF_G)}(G,S) = D_k(G)$, see \cite[Remark 12.2]{CL13}.
    \end{remark}

    \begin{corollary}
        Let $M$ be an indecomposable $V(\calF_G)$-endotrivial $kG$-module, that is, a cap of a strongly capped endo-$p$-permutation module. $M$ has an endosplit $p$-permutation resolution if and only if $[M] \in D^\Omega_k(G) + T_{V(\calF_G)}(G,S)$.
    \end{corollary}
    \begin{proof}
        The forward direction follows immediately from the previous theorem. For the reverse direction, the previous theorem implies that $[M]$ has a representative which has an endosplit $p$-permutation resolution. However, any representative of $[M]$ can be written as $M \oplus P$, where $P$ is $V(\calF_G)$-projective. If $M\oplus P$ has an endosplit $p$-permutation resolution, $M$ does as well, and the result follows.
    \end{proof}

    \begin{corollary}
        Let $G$ be a $p$-group. $\rk_\Z \calE_k(G) = c(G)$, where $c(G)$ is the number of conjugacy classes of cyclic subgroups of $G$.
    \end{corollary}
    \begin{proof}
        This follows from the short exact sequence in Theorem \ref{shortexactsequence}. It is demonstrated in Theorem \ref{thm:basisofcalEVFG} that $\rk_\Z \calE_k^{V(\calF_G)}(G)$ is equal to the number of conjugacy classes of subgroups of $G$, and the classification of $D_k(G)$ for $p$-groups states that $\rk_\Z D_k(G)$ is equal to the number of conjugacy classes of non-cyclic subgroups of $G$, see \cite[Corollary 12.9.11]{Bou10}.
    \end{proof}

    We have another homomorphism, the Bouc homomorphism, which has image in $D^\Omega_k(G)$ and which will be critical in the sequel. This homomorphism was defined for non-$p$-groups by Gelvin and Yal\c{c}in in \cite[Theorem 1.4]{GeYa21}

    \begin{definition}\label{def:bouchom}
        Let $G$ be an finite group. $\Psi: CF(G,p) \to D^\Omega(G)$ is defined as follows. For the basis of $ CF(G,p)$ given by $\{\omega_Q\}_{Q \in [s_p(G)]}$, we $\Z$-linearize the assignment
        \[\omega_Q \mapsto \begin{cases}\Omega_Q & Q \not\in \Syl_p(G) \\  [k] & Q \in \Syl_p(G) \end{cases}\] This homomorphism is called the \textit{Bouc homomorphism}.
    \end{definition}

    Although the Bouc homomorphism is defined by $\Z$-linearization for transitive $G$-sets, the definitions of $\omega_X$ and $\Omega_X$ ensure that the Bouc homomorphism behaves as expected for all $G$-sets.

    \begin{prop}{\cite[Proposition 6.9]{GeYa21}}
        For all $G$-sets $X$, $\Psi(\omega_X) = \Omega_X$.
    \end{prop}

    \section{The case of $p$-groups} \label{section:caseofpgroups}

    We are ready to begin the classification of endotrivial complexes, focusing first on the case of $p$-groups. For this section, unless we specify otherwise, we assume $G$ is a $p$-group. In this case, the short exact sequence in Theorem \ref{shortexactsequence} can be simplified as follows:

    \[0 \to \calE_k(G) \to \calE_k^{V(\calF_G)}(G) \xrightarrow{\calH} D^\Omega(G) \to 0.\]

    In addition, we have an isomorphism $h: \calE_k^V(G) \cong CF(G)$ induced by the assignment $C_Q \mapsto \omega_Q$. The picture is as follows:

    \begin{figure}[H]
        \centering
        \label{shortexactseqforendotrivs}
        \begin{tikzcd}
            0 \ar[r] & \calE_k(G) \ar[r] & \calE_k^{V(\calF_G)}(G) \ar[r, "\calH"] \ar[d, "h"'] & D^\Omega(G) \ar[r] & 0 \\
            &  & CF(G) \ar[ru, "\Psi"]
        \end{tikzcd}

    \end{figure}

    The Bouc homomorphism is the unique homomorphism which makes the above diagram commute. The next proposition verifies this.

    \begin{prop}\label{prop:bouchomcommutesforpgroups}
        If $G$ is a $p$-group, $\calH = \Psi\circ h$.
    \end{prop}
    \begin{proof}
        Note in this case $\calE_k^{V(\calF_G)}(G)$ is a free abelian group. It suffices to show commutativity for a $\Z$-basis of $\calE_k^{V(\calF_G)}(G)$. Recall a basis is given by the set \[\{C_Q = k[G/Q] \to k\}_{Q \in [s_p(G)] \setminus G} \cup \{C_G = k[1]\}.\] In Definition \ref{basisnotationdefinitions}, we observed $\calH(C_Q) = \Omega_{G/Q}$ when $Q \neq G$ and $\calH(C_G) = [k]$. For the other side, if $Q \neq G$, $(\Psi \circ h)(C_Q) = \Psi(\omega_Q) = \Omega_Q$, and $(\Psi\circ h)(C_G) = \Psi(\omega_G) = [k]$. Thus the diagram commutes.
    \end{proof}

    \begin{remark}
        In general, this equality will not hold if $G$ is not a $p$-group (and in fact, one must enlarge the codomain to $D^\Omega_k(G) + T_{V(\calF_G)}(G,S))$. For instance, for any $M \in T_{V(\calF_G)}(G,S)$, we have \[\calH(M[0]) = [M] \in D^\Omega(G) + T_{V(\calF_G)}(G,S),\] but \[\Psi\circ h(M[0]) = [k] \in D^\Omega(G) + T_{V(\calF_G)}(G,S).\] There is a subgroup of $\calE_k^{V(\calF_G)}(G)$ for which equality holds however, which we discuss in the sequel.
    \end{remark}

    In \cite{BoYa06}, Bouc and Yal\c{c}in proved that the Bouc homomorphism can be regarded as a surjective natural transformation of biset functors, and computed its kernel, which we now describe.

    \begin{definition}
        \begin{enumerate}
            \item The assignment $G \mapsto D^\Omega(G)$ for all $p$-groups $G$ defines a rational $p$-biset functor $D^\Omega$. We refer the reader to \cite[Section 12.6]{Bou10} for details on this construction.
            \item The assignment $G\mapsto CF(G)$ (resp. $G \mapsto B(G)^*$) for all $p$-groups $G$ defines isomorphic rational $p$-biset functors $C\cong B^*$. Given a $(H,G)$-biset $U$, $CF(U)$ is defined to be the morphism \[T_U: f \mapsto \left(K \mapsto \sum_{u \in K \backslash U / G} f(K^u)\right),\] where $K \leq H$, $[K \backslash U / G]$ is a set of representatives of $(K, G)$-orbits on $U$, and $K^u$ is the subgroup of $G$ defined by \[K^u = \{g\in G \mid \exists k \in K, ku = ug\}.\]
        \end{enumerate}

    \end{definition}

    \begin{definition}{\cite[Page 210]{tD87}}\label{def:bsf}
        Let $G$ be an finite group. Say $(T,S)$ is a \textit{subquotient} of $G$ if $S \trianglelefteq T \leq G$. A superclass function on $G$ valued on $p$-subgroups is a \textit{Borel-Smith function} if it satisfies the following conditions, called the \textit{Borel-Smith conditions}.
        \begin{itemize}
            \item If $(T,S)$ is a subquotient of $p$-subgroups of $G$ such that $T/S \cong (\Z/p\Z)^2$, then \[h(S) - \sum_{S < Y < T}h(Y) +ph(T) = 0.\]
            \item If $(T,S)$ is a subquotient of $p$-subgroups of $G$ such that $T/S$ is cyclic of order $p$, for an odd prime $p$, or cyclic of order 4, then $h(S) \equiv h(\hat{S}) \mod 2$, where $\hat{S}/S$ is the unique subgroup of prime order of $T/S$.
            \item If $(T,S)$ is a subquotient of $p$-subgroups of $G$ such that $T/S$ is a quaternion group of order $8$, then $h(S) \equiv h(\hat{S}) \mod 4$, where $\hat{S}/S$ is the unique subgroup of order 2 of $T/S$.
        \end{itemize}
        The Borel-Smith functions of $G$ form an additive subgroup of $CF(G,p)$, denoted by $CF_b(G,p)$. Moreover, for $p$-groups, the assignment $G \mapsto CF_b(G)$ forms a rational $p$-biset biset subfunctor of $CF$. See \cite[Proposition 3.7]{BoYa06} for details.
    \end{definition}

    Borel-Smith functions are also defined for arbitrary superclass functions by requiring that the Borel-Smith conditions are satisfied over all sections of the above form, rather than only sections of $p$-subgroups. However for the context of this paper, we are only concerned with superclass functions valued on $p$-subgroups. For the following theorem, in the original paper, the biset functor $CF$ is replaced with $B^*$, however they are canonically identified (see Remark \ref{rem:identification}).

    \begin{theorem}{\cite[Theorem 1.2]{BoYa06}}
        The kernel of $\Psi: CF \to D^\Omega$ is the biset functor $CF_b$ of Borel-Smith functions. Therefore, there is an exact sequence of $p$-biset functors of the form \[0 \to CF_b \to CF \to D^\Omega \to 0.\]
    \end{theorem}

    Recall that since $G$ is a $p$-group, $\calE_k^{V(\calF_G)}(G)$ has a canonical basis $\{C_Q\}_{Q\in [s_p(G)]}$ as described in Definition \ref{basisnotationdefinitions}.

    \begin{theorem}\label{bigtheorem}
        Let $G$ be a $p$-group. We have an isomorphism of short exact sequences:

        \begin{figure}[H]
            \centering
            \begin{tikzcd}
                0 \ar[r] & \calE_k(G) \ar[r] \ar[d, "h"] & \calE_k^{V(\calF_G)}(G) \ar[r, "\calH"] \ar[d, "h"] & D^\Omega(G) \ar[r] \ar[d, "="] & 0\\
                0 \ar[r] & CF_b(G) \ar[r]  & CF(G) \ar[r, "\Psi"] & D^\Omega(G) \ar[r] & 0
            \end{tikzcd}
            \label{fig:bigshortexactsequence}
        \end{figure}

        In particular, $[C] \in \calE_k^{V(\calF_G)}(G)$ has an endotrivial cap if and only if $h_C$ is a Borel-Smith function, and we have an isomorphism $h: \calE_k(G) \to CF_b(G)$.
    \end{theorem}
    \begin{proof}
        The commutativity of the right-hand square is precisely the statement of Proposition \ref{prop:bouchomcommutesforpgroups}, and $h$ is an isomorphism by Theorem \ref{thm:basisofcalEVFG}. Therefore, we have a short exact sequence isomorphism
        \begin{figure}[H]
            \centering
            \begin{tikzcd}
                0 \ar[r] & \ker (\calH) \ar[r] \ar[d, "h"] & \calE_k^{V(\calF_G)}(G) \ar[r, "\calH"] \ar[d, "h"] & D^\Omega(G) \ar[r] \ar[d, "="] & 0\\
                0 \ar[r] & \ker (\Psi) \ar[r]  & CF(G) \ar[r, "\Psi"] & D^\Omega(G) \ar[r] & 0
            \end{tikzcd}
        \end{figure}
        Since $\ker(\calH) = \calE_k(G)$ and $\ker(\Psi) = CF_b(G)$, the result follows.
    \end{proof}

    \begin{remark}
        Using the construction presented in Theorem \ref{thm:hmarktocomplex}, given any Borel-Smith function $f \in CF_b(G)$ for $G$ a $p$-group,
        we can construct a corresponding $V(\calF_G)$-endosplit-trivial complex which contains
        as a direct summand the unique indecomposable endotrivial complex with h-mark function $f$. However, there does not appear an easy way to construct the unique indecomposable endotrivial complex with h-mark function $f$ without taking direct summands. Determining an explicit construction for these complexes is of great interest. 
    \end{remark}

    \begin{remark}
        The isomorphism constructed in Theorem \ref{bigtheorem} factors through an injective morphism of short exact sequences, where the intermediate short exact sequence was constructed by Yal\c{c}in's study of $G$-Moore complexes, see \cite{Ya17} for definitions and details. Indeed, there is a unique injective group homomorphism $\mathscr{Y}$ which sends for any $Q \in s_p(G)\setminus \{G\}$ the chain complex $C_Q$ to the equivalence class in $\calM_G$ of the $G$-Moore space $G/Q$, regarded as a discrete $G$-CW-complex, and sends $C_G$ to the discrete $G$-Moore space with two points and trivial $G$-action. This leads to the following morphisms of short exact sequences:

        \begin{figure}[H]
            \centering
            \begin{tikzcd}
                0 \ar[r] & \calE_k(G) \ar[r] \ar[d, "\mathscr{Y}"] & \calE_k^{V(\calF_G)}(G) \ar[r, "\calH"] \ar[d, "\mathscr{Y}"] & D^\Omega(G) \ar[r] \ar[d, "="] & 0\\
                0 \ar[r] & \calM_0(G) \ar[r] \ar[d, "\Dim"]  & \calM(G) \ar[r, "\mathfrak{hom}"] \ar[d, "\Dim"] & D^\Omega(G) \ar[r] \ar[d, "="] & 0\\
                0 \ar[r] & CF_b(G) \ar[r]  & CF(G) \ar[r, "\Psi"] & D^\Omega(G) \ar[r] & 0
            \end{tikzcd}

        \end{figure}

        It follows from definitions that $h = \Dim \circ \mathscr{Y}$ and that the diagram commutes. The author thanks Erg\"un Yalcin for bringing to his attention these $G$-Moore complexes.

    \end{remark}

    In fact, if $G$ is a $p$-group, the group $CF_b(G)$ of Borel-Smith functions, and hence the group $\calE_k(G)$ of endotrivial complexes, has a canonical basis with a particularly nice property. This was known before (see the references provided in \cite{tD79}), but we spell it out in detail. 

    \begin{definition}
        We say that a superclass function $f$ is increasing (resp. decreasing) if the following holds: if $K \leq H$ are subgroups of $G$, then $f(K) \leq f(H)$ (resp. $f(K) \geq f(H)$). If either holds, we say that the function is \textit{monotone}. Similarly, we say that an endotrivial complex $C$ is increasing (resp. decreasing) if its corresponding h-mark function $h_C$ is increasing (resp. decreasing). 
    \end{definition}
    
    Given an $kG$-module $V$, the \textit{dimension function} associated to $V$ is the superclass function \[\dim: H \mapsto \dim_k V^H.\] If $k$ has characteristic 0, this induces a group homomorphism $R_k(G) \to CF(G)$. 
    
    \begin{theorem}{\cite[Theorem 5.4 and Theorem 5.13, pages 211 and 216]{tD79}}
        Let $G$ be a $p$-group. The image of $\dim: R_\R(G) \to CF(G)$ is the group of Borel-Smith functions $CF_b(G)$. Moreover, if $f$ is a decreasing Borel-Smith function, there exists a real representation $V$ for which $\dim(V) = f$. 
    \end{theorem}
    
    \begin{corollary}
        Let $G$ be a $p$-group group and suppose $f$ is a Borel-Smith function. Then $f$ can be expressed as the difference of two decreasing Borel-Smith functions. In particular, every endotrivial complex is isomorphic in $K^b({}_{kG}\triv)$ to the product of a decreasing endotrivial complex and an increasing endotrivial complex (equivalently, the dual of a decreasing endotrivial complex). 
    \end{corollary}
    \begin{proof}
        This follows from \cite[Theorem 5.4, page 211]{tD79} since the dimension function of any representation is decreasing.
    \end{proof}

    \begin{theorem}
        Let $V_1, \dots, V_n$ denote the irreducible real representations of $G$. Then the set of associated dimension functions $\{f_1,\dots, f_n\}$, after removing duplicates, forms a $\Z$-basis of $CF_b(G)$. In particular, this also determines an associated canonical $\Z$-basis of $\calE_k(G)$.
    \end{theorem}
    \begin{proof}
        Since the image of the dimension homomorphism is precisely $CF_b(G)$, it suffices to show that the set $\calB := \{f_1, \dots, f_n\}$ is linearly independent after removing duplicates.  \cite[Proposition 5.9, page 213]{tD79} asserts that $\ker(\dim)$ is generated by elements of the form $V - \psi^k(V)$, where $V$ is an irreducible real representation, $\psi^k$ is the $k$-th Adams operation, and $k$ is coprime to $|G|$. Therefore, the duplicates in $\calB$ arise from Adams operation conjugates, and it follows that any set of real irreducible representations which are not Adams operation conjugates will correspond to a linearly independent set of Borel-Smith functions.  
    \end{proof}

    \begin{remark}
        The bases given for $\calE_k(G)$ in \cite[Section 6]{SKM23} for dihedral, semidihedral, and quaternion 2-groups, and elementary abelian $p$-groups, coincide exactly with the canonical basis of $\calE_k(G)$ detailed here. Moreover, the endotrivial complexes used to generate the twisted cohomology ring in \cite[Section 3]{BG23} is also precisely the canonical basis of $\calE_k(G).$ Those bases were determined in an ad-hoc manner, suggesting that in fact, there is some heuristic for constructing these decreasing canonical basis elements of $\calE_k(G).$

    \end{remark}

    \section{Transporting rational $p$-biset functor structure} \label{sec:bisetstructure}

    The bottom row of the isomorphism of short exact sequences in Theorem \ref{bigtheorem} arises from the short exact sequence of rational $p$-biset functors shown in \cite[Theorem 1.2]{BoYa06}, \[0 \to CF_b \to B^* \to D^\Omega \to 0,\] where $B^*$ is identified with $CF$. We can transport the structure of those biset functors to the top row, thus realizing $\calE_k^{V(\calF)}$ as a rational $p$-biset functor canonically isomorphic to the rational $p$-biset functor $C$ and $\calE_k(G)$ as a subfunctor which is canonically isomorphic to the rational $p$-biset functor $CF_b$, both via the h-mark homomorphism $h$.

    For this section, since we are concerned only with rational $p$-biset functors, all groups considered are finite $p$-groups.

    \begin{definition}
        Let $\calC$ denote the full subcategory of the biset category consisting of $p$-groups.
        \begin{enumerate}
            \item We define the biset functor \[\calE_k^{V(\calF)}: \calC \to \mathbf{Ab}\] by the assignment $G \mapsto \calE_k^{V(\calF_G)}(G)$. For any $(H,K)$-biset $U$, the action of $U$ on $\calE_k^{V(\calF_K)}(K)$ is given by \[h_H\inv \circ T_U \circ h_K,\] where $h_K: \calE_k^{V(\calF_K)}(K) \xrightarrow{\sim} CF(K)$ and $h_H: \calE_k^{V(\calF_H)}(H) \xrightarrow{\sim} CF(H)$ denote the h-mark isomorphisms over $K$ and $H$ respectively and $T_U$ is the generalized induction defined for $C$. Since $C$ is a rational $p$-biset functor and  $\calE_k^{V(\calF)}\cong C$ by construction, $\calE_k^{V(\calF)}$ is a rational $p$-biset functor.

            \item We define the biset subfunctor \[\calE_k: \calC \to \mathbf{Ab}\] of $\calE_k^{V(\calF)}$ by the assignment $G \mapsto \calE_k(G)$. The biset action on $\calE_k$ is the same as before, since $CF_b$ is a subfunctor of $C$. 
        \end{enumerate}

    \end{definition}

    In \cite[Proposition 4.3]{SKM23}, we determined what effects restriction, inflation, and the Brauer construction have on h-marks of endotrivial complexes. For $\calE_k$, the biset functor structure coincides with those computations.

    \begin{theorem}
        Let $H \leq G$ and $N \trianglelefteq G.$
        \begin{enumerate}
            \item Consider the inflation biset $\text{inf}_{G/N}^G := {}_{G}G/N_{G/N}$. The operation \[\text{inf}_{G/N}^G: \calE_k(G/N) \to \calE_k(G)\] coincides with the usual inflation homomorphism $\Inf_{G/N}^G: \calE_k(G/N) \to \calE_k(G)$.

            \item Consider the restriction biset $\text{res}_{H}^G:= {}_{H}G_G$. The operation \[\text{res}_{H}^G: \calE_k(G) \to \calE_k(H)\] coincides with the usual restriction $\Res_{H}^G: \calE_k(G) \to \calE_k(H)$.

            \item Consider the deflation biset $\text{def}_{G/N}^G := {}_{G/N}G/N_{G}.$ The operation \[\text{def}_{G/N}^G: \calE_k^{V(\calF_{G})}(G) \to \calE_k^{V(\calF_{G/N})}(G/N)\] coincides with the Brauer construction $-(N): \calE_k(G) \to \calE_k(G/N)$.
        \end{enumerate}
    \end{theorem}
    \begin{proof}
        In \cite[Proposition 4.3]{SKM23}, the effects of restriction, inflation, and the Brauer construction on h-marks are computed. It suffices to show that these biset operations coincide with those computations.
        \begin{enumerate}
            \item Let $L \leq G$ and $f \in CF(G/N)$. We compute:
            \begin{align*}
                (T_{\text{inf}^G_{G/N}} f)(L) &=  \sum_{u \in [L\backslash {}_{G}G/N_{G/N} / G/N]} f(L^u)\\
                \intertext{We set $u = N \in G/N$ since $\inf^G_{G/N}$ is transitive as right $G/N$-set. Then, $L^u = \{gN \in G/N \mid \exists l\in L, lN = gN\} = LN/N$.}
                &= f(LN/N).
            \end{align*}
            This coincides with the description of \cite[Proposition 4.3(c)]{SKM23}, which shows that the h-mark homomorphism of $\Inf^G_{G/N} C$ is $h_{\Inf^G_{G/N} C}(H) = h_{C}(HN/N)$ for any $H \leq G$.

            \item Let $L \leq H$ and $f \in CF(G)$. We compute:
            \begin{align*}
                (T_{\text{res}^G_H} f)(L) &= \sum_{u \in [L\backslash {}_HG_G/G]} f(L^u)\\
                \intertext{We set $u = 1 \in G$, since $\text{res}^G_H$ is transitive as a right $G$-set. Then, $L^1 = \{g \in G \mid \exists l\in L, l=g\}$.}
                &= f(G/L).
            \end{align*}
            This coincides with the description of \cite[Proposition 4.3(a)]{SKM23}, which shows that the h-mark homomorphism of $\Res^G_H C$ is the restriction $\Res^G_H h_C$.

            \item Let $L/N \leq G/N$ and $f \in CF(G)$. We compute:
            \begin{align*}
                (T_{\text{def}^G_{G/N}} f)(L/N) &= \sum_{u \in [(L/N)\backslash {}_{G/N}G/N_G / G]} f((L/N)^u)\\
                \intertext{We set $u = N \in G/N$ since $\text{def}^G_{G/N}$ is transitive on the right. Then, $(L/N)^{N} = \{g \in G \mid \exists lN \in L/N, lN = gN\} = L$, since $N \trianglelefteq L$.}
                &= f(L).
            \end{align*}
            This coincides with the description of \cite[Proposition 4.3(b)]{SKM23}, which shows the h-mark homomorphism of $C(N)$ is $h_{C(N)}(L/N) = h_C(L)$ for any $L \trianglerighteq N$.
        \end{enumerate}
    \end{proof}

    \begin{remark}
        Informally, the endotrivial complex arising from a real representation $V$ arises from the Bredon cohomology complex associated to the representation sphere $S^V$, see for instance \cite{Br67}. Indeed, the h-mark at $H\leq G$ of an endotrivial complex $C$ associated to a real representation $V$ counts the dimension of $V^H$, or equivalently, of $S^{(V^H)}$. Since the dimension function $\dim: R_\R(G) \to CF_b(G)$ is a biset functor, it follows that the induction biset $\text{ind} = {}_GG_H$ corresponds to the topological norm map $N^G_H$, see for instance \cite[Section A.4]{HHR16}. Indeed, we have an isomorphism \[N^G_H S^V \cong S^{\Ind^G_H V},\] see for instance \cite[Proposition A.59]{HHR16}. The author greatly thanks Martin Gallauer for this observation and reference. 
    
        For class functions, we have a formula \[(T_{\text{ind}^G_H} f)(L) = \sum_{x \in [L\backslash G/ H]} f(H\cap L^x).\]
        We note induction does not correspond to the usual tensor induction of chain complexes (see \cite[Section 4.1]{Be982} for a description, as well as a description of the norm map for Ext). In fact, tensor induction of chain complexes does not preserve homotopy equivalence, therefore it does not preserve endotriviality in general. See \cite[Remark 5.2]{SKM23} for a basic example.

    \end{remark}

    \begin{remark}
        By the exact same computations, we can see what effect the restriction,
        inflation, and deflation bisets have on the h-marks of a
        $V(\calF_G)$-endosplit-trivial chain complex.
        These, however, do not coincide with the usual maps $\Res^G_H$,
        $\Inf^G_{G/N}$, or $-(N)$ on the group $\calE_k^{V(\calF_G)}(G)$.
        For instance, the restriction biset $\text{res}^G_H$ induces
        a group homomorphism
        \[\text{res}^G_H: \calE_k^{V(\calF_G)}(G) \to \calE_k^{V(\calF_H)}(H),\]
        but the group homomorphism induced by restriction of modules is of the form
        \[\Res^G_H: \calE_k^{V(\calF_G)}(G) \to \calE_k^{\Res^G_H V(\calF_G)}(H).\]
        If $H$ does not contain a Sylow $p$-subgroup of $G$ as a subgroup,
        ${}_{kH}Proj(\Res^G_H V) = {}_{kG}\catmod$, and the restriction of a
        $V(\calF_G)$-endosplit-trivial chain complex may not be
        $V(\calF_H)$-endosplit trivial.
        For an example, if $G = D_{16}$, the chain complex
        \[C_{H} = k[D_{16}/H] \to k,\] where $H$ is any subgroup of order $4$
        satisfying $Z(D_{16}) \leq H$ and the differential is the augmentation homomorphism, satisfies
        \[\Res^{D_{16}}_{Z(D_{16})} C_H = k\oplus k \oplus k \oplus k \to k,\]
        which is not a $V(\calF_{C_2})$-endosplit-trivial complex.
        On the other hand, one computes via h-marks that
        \[\operatorname{res}^G_{Z(D_{16})} C_H = C_{Z(D_{16})}.\]

        A similar issue arises for deflation, since in general, $\big(V(\calF_G)\big)(N)$ is not an absolutely $p$-divisible $kG$-module. For instance, if $G = C_4$, then \[V(\calF_G) = k[C_4/1] \oplus k[C_4/C_2],\] therefore \[\big(V(\calF_G)\big)(C_2) \cong k\oplus k.\]

        Inflation has a slightly different issue which arises. In general, inflation preserves absolute $p$-divisibility, but ${}_{kG}Proj(\Inf^G_{G/N} V(\calF_{G/N})) \subseteq {}_{kG}Proj(V(\calF_G))$. For instance, let $G = V_4$ with 3 nonconjugate subgroups of order 2 $H_1, H_2, H_3$. Then, \[V(\calF_{G/H_1}) = k[(G/H_1)/(H_1/H_1)], \text{ and } \Inf^G_{G/H_1} V(\calF_{G/H_1}) = G/H_1.\] On the other hand, \[V(\calF_G) = k[G/H_1] \oplus k[G/H_2] \oplus k[G/H_3] \oplus k[G/1].\] Therefore, $k[G/H_2]$ and $k[G/H_3]$ are $V(\calF_G)$-projective but not $\Inf^G_{G/H_1} V(\calF_{G/H_1})$-projective.
    \end{remark}

    \section{Results for arbitrary finite groups}

    Using the results of the previous sections, we consider the case of $G$
    not necessarily a $p$-group. We complete the classification of endotrivial complexes and determine the kernel of the Bouc homomorphism. 

    \subsection{The classification of endotrivial complexes}

    \begin{definition}
        Let $V$ be a (possibly zero) absolutely $p$-divisible $kG$-module and let $H \leq G$ be a subgroup containing a Sylow $p$-subgroup $S$. We define the \textit{$G$-stable} stable subgroup $\calE_k^{V}(H)^G \leq \calE_k^V(H)$ as follows. \[\calE_k^{V}(H)^G = \{[C] \in \calE_k^V(H) \mid \text{For all } P,Q\leq H \text{ with } P =_G Q, h_C(P) = h_C(Q) \}.\]
    \end{definition}

    \begin{theorem}\cite[Theorem 12.6]{SKM24a}\label{thm:imresmap}
        Let $S \in \Syl_p(G)$. Then the group homomorphism $\Res^G_S: \calE_{k}(G) \to \calE_k(S)^G$ is surjective. Moreover, we have a split exact sequence of abelian groups \[0 \to \Hom(G.k^\times) \to \calE_k(G) \xrightarrow{\Res^G_S} \calE_k(S)^G \to 0.\]
    \end{theorem}

    Since this is a fairly nontrivial and consequential result, we reprove this theorem here for the reader's convenience. 

    \begin{lemma}\cite[Theorem 12.3]{SKM24a}\label{lem:inductionbrauercommute}
        Let $P\in s_p(G)$, $H \leq G$, and $M$ be a $p$-permutation $kH$-module. We have the following natural isomorphism of $kN_G(P)$-modules, regarding the Brauer construction as a functor $-(P): {}_{kG}\triv \to {}_{kN_G(P)}\triv$: \[\left(\Ind^G_H M\right)(P) \cong \bigoplus_{x \in [N_G(P)\backslash G / H], P\leq {}^xH} \Ind^{N_G(P)}_{N_G(P)\cap {}^xH}\big( ({}^xM)(P) \big).\]
        In particular, since the isomorphism is natural, it also holds in the category $K^b({}_{kN_G(P)}\triv)$.
    \end{lemma}

    \begin{proof}
        First, we have $(\Ind^G_H M)(P) = (\Res^G_{N_G(P)} \Ind^G_H M)(P)$. Applying the Mackey formula yields the following natural isomorphism of $kN_G(P)$-modules:

        \begin{align*}
            (\Res^G_{N_G(P)} \Ind^G_H M)(P) &\cong \left(\bigoplus_{x \in [N_G(P)\backslash G / H]} \Ind^{N_G(P)}_{N_G(P) \cap {}^xH} \Res^{{}^xH}_{N_G(P) \cap {}^xH} ({}^xM)\right)(P)\\
            &\cong \bigoplus_{x \in [N_G(P)\backslash G / H]} \left(\Ind^{N_G(P)}_{N_G(P) \cap {}^xH} \Res^{{}^xH}_{N_G(P) \cap {}^xH}({}^xM)\right)(P)
        \end{align*}
        Given a double coset representative $x$, each indecomposable constituent of $\Ind^{N_G(P)}_{N_G(P) \cap {}^xH} \Res^{{}^xH}_{N_G(P) \cap {}^xH} ({}^xM)$ has a vertex contained in $N_G(P)\cap {}^xH$, so if $\left(\Ind^{N_G(P)}_{N_G(P) \cap {}^xH} \Res^{{}^xH}_{N_G(P) \cap {}^xH} ({}^xM)\right)(P) \neq 0$, then $N_G(P) \cap {}^xH \geq_{N_G(P)} P$, and since $P \trianglelefteq N_G(P)$, this occurs if and only if $P \leq {}^xH$. Therefore, all terms in the direct sum that do not satisfy $P \leq {}^xH$ are zero, and we have
        \[(\Ind^G_H M)(P) \cong \bigoplus_{x \in [N_G(P)\backslash G / H], P\leq {}^xH} \left(\Ind^{N_G(P)}_{N_G(P) \cap {}^xH} \Res^{{}^xH}_{N_G(P) \cap {}^xH}({}^xM)\right)(P).\]
        We now compute the internal term when $P \leq {}^xH$. First, since $P \trianglelefteq N_G(P)$,
        \[ \Ind^{N_G(P)}_{N_G(P) \cap {}^xH} ({}^xM)^P = \left(\Ind^{N_G(P)}_{N_G(P) \cap {}^xH} \Res^{{}^xH}_{N_G(P) \cap {}^xH}{}^xM\right)^P.\]
        Indeed, for any $g \in N_G(P)$, and $m \in {}^xM$, $p\cdot (g\otimes m) = g\otimes m $ if and only if $g \otimes p^g m = g\otimes m$ for all $p \in P$, and the equality follows from there.
        We next claim \[\sum_{Q < P}\tr^P_Q \left(\Ind^{N_G(P)}_{N_G(P) \cap {}^xH} \Res^{{}^xH}_{N_G(P) \cap {}^xH} {}^xM\right)^Q = \Ind^{N_G(P)}_{N_G(P) \cap {}^xH} \left(\sum_{Q < P} \tr^P_Q ({}^xM)^Q\right).\]
        Indeed, for any $Q < P$, $g \otimes m \in  \left(\Ind^{N_G(P)}_{N_G(P) \cap {}^xH} \Res^{{}^xH}_{N_G(P) \cap {}^xH} {}^xM\right)^Q$ if and only if $m \in (^xM)^{Q^g}$. Now given coset representatives $p_1, \dots, p_k$ of $P/Q$, $(p_1 + \dots + p_k)g \otimes m = g \otimes (p_1^g + \dots + p_k^g)m$, and since $p_1^g, \dots,  p_k^g$ are coset representatives of $P/Q^g$, we have an equality of sets. Finally, for arbitrary groups $H\leq G$ and appropriate modules $N \subseteq M$, $\Ind^G_H M/N \cong \Ind^G_H M / \Ind^G_H N$ naturally by exactness of additive induction. We now compute, assuming $P \leq {}^xH$ and setting $N_{{}^xH}(P) = N_G(P)\cap {}^xH$:

        \begin{align*}
            \left(\Ind^{N_G(P)}_{N_{{}^xH}(P)} \Res^{{}^xH}_{N_{{}^xH}(P)}({}^xM)\right)(P) &= \left(\Ind^{N_G(P)}_{N_{{}^xH}(P)} \Res^{{}^xH}_{N_{{}^xH}(P)}{}^xM\right)^P \\
            &/ \sum_{Q < P}\tr^P_Q \left(\Ind^{N_G(P)}_{N_{{}^xH}(P)} \Res^{{}^xH}_{N_{{}^xH}(P)} {}^xM\right)^Q \\
            &= \Ind^{N_G(P)}_{N_{{}^xH}(P)} ({}^xM)^P/\Ind^{N_G(P)}_{N_{{}^xH}(P)} \left(\sum_{Q < P} \tr^P_Q ({}^xM)^Q\right)\\
            &\cong \Ind^{N_G(P)}_{N_{{}^xH(P)}}\left(({}^xM)^P/ \left(\sum_{Q < P} \tr^P_Q ({}^xM)^Q\right) \right)\\
            &= \Ind^{N_G(P)}_{N_{{}^xH(P)}}\big(({}^xM)(P)\big)
        \end{align*}

        This completes the proof of the isomorphism for modules, and since all the isomorphisms used were natural, the result follows.
    \end{proof}

    \begin{proof}[Proof (of Theorem \ref{thm:imresmap})]
        Let $D \in \calE_k(S)^G$ and assume without loss of generality that $D$ is indecomposable. We have $\calH(D) = [k]$. Since $D$ is an endosplit $p$-permutation resolution for $k$ as $kS$-module and is $G$-stable, it follows from Lemma \ref{lem:inductionbrauercommute} that for all $p$-subgroups $P$ of $G$, $(\Ind^G_S D)(P)$ has homology concentrated in exactly one degree. Hence $\Ind^G_S D$ is an endosplit $p$-permutation resolution for $\Ind^G_S H(C) \cong k[G/S]$, and moreover, for all $P \in s_p(S)$, $h_C(P)$ is the unique value of $i$ for which $H_i((\Ind^G_S C)(P)) \neq 0$, again by Lemma \ref{lem:inductionbrauercommute}. Now, $k$ is a direct summand of $k[G/S]$, therefore there is a direct summand $C \mid \Ind^G_H D$ which is an endosplit $p$-permutation resolution for $k$ as $kG$-module (see \cite[Proposition 7.11.2]{L182}), hence an endotrivial complex, and whose h-marks coincide with those of $D$. Therefore, $\Res^G_S C$ has an indecomposable direct summand isomorphic to $D$. Thus $\calE_k(S)^G \leq \im(\Res^G_S) $, and by $G$-stability, this must be an equality. 

        For the last statement, $\ker(\Res^G_S)$ consists of all endotrivial complexes with h-marks all zero. However, this is simply $\ker(h)$, which is the torsion subgroup of $\calE_k(G)$, isomorphic to $\Hom(G,k^\times)$. Therefore, the inclusion $\Hom(G,k^\times) \to \calE_k(G)$ splits with retraction $\calE_k(G) \to \Hom(G,k^\times)$ given by $\calH: [C] \mapsto H_{h_C(1)}(C) \in \Hom(G,k^\times)$.

    \end{proof}

    Armed with this crucial fact, we are now ready to complete the classification. 

    \begin{corollary}\label{thm:borelsmithfornonpgroups}
        Let $G$ be a finite group. We have a split exact sequence of abelian groups \[0 \to \Hom(G,k^\times) \to \calE_k(G) \xrightarrow{h} CF_b(G,p) \to 0.\] 
    \end{corollary}
    \begin{proof}
        \cite[Remark 3.8]{SKM23} states that we have a split exact sequence \[0 \to \Hom(G,k^\times) \to \calE_k(G) \xrightarrow{h} \im(h) \to 0,\] with retract given by $\calH: \calE_k(G) \to \Hom(G,k^\times)$, so it suffices to determine $\im(h).$ We show $\im(h) = CF_b(G,p).$ 
        
        Since $\Res^G_S: \calE_k(G) \to \calE_k(S)^G$ is surjective, given any Borel-Smith function $f \in CF_b(G,p)$ and a Sylow $p$-subgroup $S$ of $G$, $\Res^G_S f$ is a $G$-stable Borel-Smith function for $S$. By Corollary \ref{bigtheorem}, there exists an endotrivial complex $D$ of $kS$-modules with $h_D = \Res^G_S f$, and by surjectivity of restriction, there exists an endotrivial complex $C$ of $kG$-modules with $[\Res^G_S C] = [D]$ in $\calE_k(S)$. Therefore, $h_C = f$, so $CF_b(G,p) \leq \im(h).$ Now, suppose $C$ is an endotrivial complex of $kG$-modules, with h-mark function $h_C$. Then $\Res^G_S h_C$ is necessarily a Borel-Smith function, but observe that $\Res^G_S: CF_b(G,p) \to CF_b(S)^G$ is in fact an isomorphism, since every $p$-subgroup of $G$ is conjugate to a subgroup of $S$. Therefore, $h_C$ is a Borel-Smith function, and we have the equality $CF_b(G,p) = \im(h)$.  
    \end{proof}

    \subsection{The kernel of the Bouc homomorphism}

    \begin{remark}
        Theorem \ref{thm:borelsmithfornonpgroups} gives us all possible h-marks which arise from endotrivial complexes, and Theorem \ref{thm:hmarktocomplex} gives a recipe for reconstructing an endotrivial complex from its h-marks, up to a twist by a strongly capped indecomposable endo-$p$-permutation $kG$-module. In the $p$-group case, this information tells us how to reconstruct a complex, since the only strongly capped indecomposable endo-$p$-permutation $kG$-module is $k$. However, in the non-$p$-group case, the situation is less clear.

        Recall in Theorem \ref{thm:hmarktocomplex}, the setup was as follows: for $P \in [s_p(G)]$, we set \[b_P:= \sum_{Q \in [s_p(G)]} h_{C}(Q)\mu(P,Q).\]

        Then, the chain complex \[D := \bigotimes_{Q \in [s_p(G)]} C_Q^{b_Q}\] is $V(\calF_G)$-endosplit-trivial. One goal is to determine when $D$ has an endotrivial summand, and if not, a $M \in T_{V(\calF_G)}(G,S)$ for which $M[0] \otimes_k D$ has an endotrivial summand. An even stronger question we can ask is whether $D \in \ker (\calH): \calE_k^{V(\calF_G)}(G) \to D^\Omega_k(G)+T_{V(\calF_G)}(G,S).$ Observe that $D$ is generated by a product of chain complexes representing relative syzygies. This motivates the following definition.
    \end{remark}

    \begin{definition}\label{def:subgroupgenbyrelsyzygies}
        Define the subgroup ${}^\Omega\calE_k^{V(\calF_G)}(G) \leq \calE_k^{V(\calF_G)}(G)$ by \[{}^\Omega\calE_k^{V(\calF_G)}(G) := \langle [C_P] \mid P \in s_p(G) \rangle.\] Define the subgroup $\calE_k^\Omega(G) \leq \calE_k(G) \leq \calE_k^{V(\calF_G)}(G)$ by \[\calE_k^\Omega(G) := {}^\Omega\calE_k^{V(\calF_G)}(G) \cap \calE_k(G).\] Finally, define $\calT\calE_k^\Omega(G) := \calT\calE_k(G) \cap \calE_k^\Omega(G)$. It is straightforward that \[\calH^\Omega: {}^\Omega\calE_k^{V(\calF_G)} \to D_k^\Omega(G)\] is a well-defined surjective group homomorphism, and moreover, $\ker(\calH^\Omega) = \calT\calE_k^\Omega(G)$. We have a short exact sequence \[0 \to \calT\calE_k^\Omega(G) \to  {}^\Omega\calE_k^{V(\calF_G)}(G) \to D^\Omega(G) \to 0,\] where $\calH^\Omega$ is the surjection.
    \end{definition}

    Rephrasing the previous remark, we wish to determine $\calE_k^\Omega(G)$ or $\calT\calE_k^\Omega(G)$. One way we can do so is by returning to the Bouc homomorphism $\Psi: CF(G,p) \to D^\Omega(G)$, in the more general case of non-$p$-groups. In this case, we no longer necessarily have an isomorphism $CF_b(G,p)\cong \ker(\Psi)$.

    Gelvin and Yal\c{c}in describe in \cite{GeYa21} constraints for $\ker(\Psi)$.

    \begin{definition}{\cite[Definitions 9.6, 9.7]{GeYa21}}\label{def:orientedartin}
        A function $f \in CF(G,p)$ satisfies the \textit{oriented Artin condition} if for any distinct prime numbers $p$ and $q$, consider $L \triangleleft K \triangleleft H \leq N_G(L)$ subgroups of $G$ such that $K$ is a cyclic $p$-group, $K/L \cong \Z/p$, and $H/K \cong \Z/q^r$. Then $f(L) \equiv f(K) \mod 2q^{r-l}$, where $H/K$ acts on $K/L$ with kernel of order $q^l$.

        A \textit{oriented Artin-Borel-Smith function} is a superclass function that satisfies the Borel-Smith conditions and the oriented Artin condition. The subgroup of $CF(G,p)$ consisting of oriented Artin-Borel-Smith functions is denoted by $CF_{ba^+}(G,p)$.
    \end{definition}

    If $p = 2$, then any Borel-Smith function vacuously satisfies the Artin condition, since there is no nontrivial automorphism of $K/L \cong \Z/2$.

    \begin{theorem}{\cite[Theorem 9.10]{GeYa21}}
        Let $G$ be a finite group and $\Psi: CF(G,p) \to D^\Omega(G)$ be the Bouc homomorphism. Then \[CF_{ba^+}(G,p) \subseteq \ker (\Psi) \subseteq CF_b(G,p).\]
        In particular, if $p = 2$, $\ker(\Psi) = CF_b(G,p)$.
    \end{theorem}

    In \cite{GeYa21}, Gelvin and Yal\c{c}in asked if in general, $CF_{ba^+}(G,p) = \ker(\Psi)$; they remark that they could find no counterexamples to the claim. Moreover, a nontrivial example was provided in which $ CF_{ba^+}(G,p) = \ker(\Psi) \subset CF_b(G,p)$.

    \begin{lemma}{\cite[Lemma 9.11]{GeYa21}}\label{lem:gelvinyalcinlemma}
        Let $Q\trianglelefteq G$ be a cyclic subgroup of $G$ of order $p$ such that $G/Q \cong \Z/q^r$. Suppose that $G/Q$ acts on $Q$ with kernel order $q^l$. Then $D^\Omega(G) \cong \Z/2q^{r-l}$ and the equality $\ker(\Psi) = CF_{ba^+}(G,p)$ holds.
    \end{lemma}

    The next proposition gives us an equivalent formulation of $\ker (\Psi)$.

    \begin{prop}\label{prop:restshortexseqiso}
        We have an isomorphism of short exact sequences:

        \begin{figure}[H]
            \begin{tikzcd}
                0 \ar[r]& \calT\calE_k^\Omega(G) \ar[r] \ar[d, "h"] & {}^\Omega\calE_k^{V(\calF_G)}(G) \ar[r, "\calH^\Omega"] \ar[d, "h"] & D^\Omega(G) \ar[r] \ar[d, "="] & 0 \\
                0 \ar[r] & \ker (\Psi) \ar[r] & CF(G,p) \ar[r, "\Psi"] & D^\Omega(G)  \ar[r]& 0
            \end{tikzcd}
            \centering
        \end{figure}
    \end{prop}
    \begin{proof}
        First, we have $\calH^\Omega = \Psi \circ h$ by a similar argument as used in the proof of Proposition \ref{prop:bouchomcommutesforpgroups}. In particular, ${}^\Omega\calE_k^{V(\calF_G)}(G)$ is free abelian with basis given by elements of the form $[C_P]$ for $P \in [s_p(G)]$. The left-hand square commutes simply since the injections are both inclusions of kernels. The image of this basis of ${}^\Omega\calE_k^{V(\calF_G)}(G)$ under $h$ is precisely the basis of $CF(G,p)$ given by elements of the form $\omega_P$. Since $[C_P]\mapsto \omega_P$, $h$ is an isomorphism. Since the top and bottom injective homomorphisms are both subgroup inclusion, $h: \calT\calE_k^\Omega(G) \to \ker(\Psi)$ is also an isomorphism.
    \end{proof}

    In fact, we use the identification given in Proposition \ref{prop:restshortexseqiso} to completely determine $\ker(\Psi)$, giving a positive answer to Gelvin and Yal\c{c}in's question. To do this, we briefly recall a character-theoretic result of Boltje and Carman regarding the orthogonal unit group of the trivial source ring $O(T(kG))$, and the relation of $O(T(kG))$ to $\calE_k(G)$ studied in \cite{SKM23}. 

    \begin{definition}
        Recall the \textit{trivial source ring} $T(kG)$ is the split Grothendieck group of ${}_{kG}\triv$ with multiplication induced by $\otimes_k$, and identically is the Grothendieck ring of the tensor-triangulated category $K^b({}_{kG}\triv).$ The \textit{orthogonal unit group} $O(T(kG))$ is the subgroup of $T(kG)^\times$ consisting of units $u \in T(kG)$ for which $u\inv = u^*$.
    \end{definition}

    \begin{theorem}{\cite[Theorem C]{BC23}}\label{thm:boltjethmc}
        Let $S$ be a Sylow $p$-subgroup of $G$ and let $\calF := \calF_S(G)$ be the associated fusion system on $S$. One has a direct product decomposition \[O(T(k G))\cong B(\calF)^\times \times \left(\prod_{P\in s_p(G)} \Hom(N_G(P)/P,k^\times)\right)'\] where the second factor is defined as the set of all tuples \[(\varphi_P) \in \left(\prod_{P\in s_p(G)} \Hom(N_G(P)/P, k^\times)\right)^G\] satisfying \[\varphi_P(xP) = \varphi_{P\langle x_p\rangle}( xP\langle x_p\rangle)\] for all $P \in s_p(G)$ and $x\in N_G(P)$.
    \end{theorem}

    Here, $x_p$ denotes the $p$-part of $x$. That is, given any $x \in G$, $x = x_p x_{p'} = x_{p'}x_p$, where $x_p$ has order a power of $p$ and $x_{p'}$ has $p'$ order.

    Taking the Lefschetz invariant of an endotrivial complex induces a group homomorphism \[\Lambda: \calE_k(G) \to O(T(kG)), [C] \mapsto \Lambda(C).\]
    Moreover, given an endotrivial complex $C$, the tuple $(\varphi_P)$ in \cite[Theorem C]{BC23} arising in the identification of $\Lambda(C) \in O(T(kG))$ corresponds to the local homology of $C$ via \[\calH_C(P) = [k_{\varphi_P}] \in D^\Omega(N_G(P)/P).\] See \cite[Proposition 4.5]{SKM23} for details.

    We are now ready to prove Gelvin and Yal\c{c}in's conjecture with endotrivial complex machinery. In fact, we do so by reducing to the case they prove in {\cite[Lemma 9.11]{GeYa21}}, then apply some character theory.

    \begin{theorem}\label{thm:bouchomker}
        We have $\ker(\Psi) = CF_{ba^+}(G,p)$.
    \end{theorem}
    \begin{proof}
        It suffices to show $CF_{ba^+}(G,p) \supseteq \ker (\Psi)$. Suppose $[C] \in \calT\calE_k^\Omega(G)$ and let $C$ be the unique indecomposable representative of $[C]$. By Proposition \ref{prop:restshortexseqiso}, it suffices to show that $h_C$ satisfies the oriented Artin-Borel-Smith conditions. For any collection of subgroups $L\triangleleft K \triangleleft H \leq N_G(L)$ as specified in the oriented Artin conditions, we will show $h_C(L) \equiv h_C(K) \mod 2q^{r-l}$. Note that $r$ and $l$ are local conditions, only depending on $K/L$ and $H/K$.

        Let $D = \Res^{N_G(L)/L}_{H/L}C(L)$. Then \[h_D(K/L) = h_C(K) \text{ and } h_D(L/L) = h_C(L)\] by \cite[Proposition 4.3]{SKM23}. Moreover, the values of $r$ and $l$ do not change. Therefore, it suffices to show $h_D$ satisfies $h_D(L/L) \equiv h_D(K/L) \mod 2q^{r-l}$. Notice that $H/L$ satisfies the conditions of \cite[Lemma 9.11]{GeYa21}, with $K/L$ as the subgroup $Q$ in the lemma. Therefore, by \cite[Lemma 9.11]{GeYa21}, if $h_D \in \ker (\Psi)$, the congruence holds, and we are done. This occurs if and only if $\calH_D(1) = [k] \in D^\Omega(H/K)$ by Proposition \ref{prop:restshortexseqiso}. Note that \[\Res^{N_G(L)/L}_{H/L} \calH_{C}(L) = \calH_D(1) \in  D^\Omega(H/K).\]
        On the other hand, \[\Res^{N_G(L)/L}_{H/L} \calH_{C}(L) = [k_{\varphi_L|_{H/L}}],\] where $\varphi_L$ is the character corresponding to $u$ from the decomposition in \cite[Theorem C]{BC23} Thus, it suffices to show that for any $x \in H$, $\varphi_L(xL) = 1.$

        Since $h_C \in \ker (\Psi)$, $\varphi_1$ is the trivial character. By \cite[Theorem C]{BC23}, we have that for any $x \in H$, \[1 = \varphi_1(x) = \varphi_{\langle x_p\rangle}(x \langle x_p\rangle)=\varphi_{\langle x_p\rangle}(x_{p'} \langle x_p\rangle).\]

        We claim that for every coset $xL$ of $H/L$, there exists a $y \in H$ for which $\varphi_L(yL) = \varphi_L(xL)$ and $y_p$ is a generator of $L$, where $y = y_py_{p'}$ is the unique decomposition of $y$ into $p$- and $p'$-parts. Indeed, if $\langle x_p\rangle = L$, $y = x$ suffices. If not, then since $K$ is a $p$-group, $x_p \in K$, and there exists some $k\in K$ for which $x_pk$ is a generator of $L \leq K$. Let $y = xk$. Then by construction, $\langle y_p\rangle =\langle x_pk\rangle = L$. Now, since $\varphi_L$ is a linear Brauer character and $k$ is a $p$-element of $H$, \[\varphi_L(yL) = \varphi_L(xkL) = \varphi_L(x_{p'}L) = \varphi_L(xL),\] as desired. Therefore, for every coset $xL\in H/L$, we have \[1 = \varphi_1(y) = \varphi_{L}(yL) = \varphi_L(xL),\] thus $\varphi_{L}$ is the trivial character, and we are done.

    \end{proof}

    \begin{corollary}
        Let $f \in CF_{b}(G,p)$. For $P \in [s_p(G)]$, set \[b_P:= \sum_{Q \in [s_p(G)]} f(Q)\mu(P,Q).\]

        The $V(\calF_G)$-endosplit-trivial chain complex
        \[C = \bigotimes_{Q \in [s_p(G)]} C_Q^{\otimes b_Q}\] has an endotrivial cap $C_0$ with  $\calH(C_0) = k$ if and only if $f \in CF_{ba^+}(G,p)$.
    \end{corollary}
    \begin{proof}
        For the reverse direction, it follows from the proof of Theorem \ref{thm:hmarktocomplex} that $h_C = f$. From Proposition \ref{prop:restshortexseqiso}, it follows that $C \in \calT\calE_k^\Omega(G)$, and therefore its cap $C_0$ has indecomposable homology isomorphic to $k$, as desired. For the forward direction, if $f \not\in CF_{ba^+}(G,p)$, then $[C] \not\in \ker (\calH)$ by the previous theorem and the isomorphism in Proposition \ref{prop:restshortexseqiso}, and the result follows.
    \end{proof}

    It remains unclear what occurs if $f \in CF_b(G,p)$ but $f\not\in CF_{ba^+}(G,p)$. One question of interest which remains is when $C$ as constructed in the previous corollary or as in Theorem \ref{thm:hmarktocomplex} has an endotrivial cap which does not belong to $\ker(\calH)$.

    \section{Extending virtual modules to invertible complexes for $p$-groups}

    In this section, we use results from \cite{BoYa06} to prove that the Lefschetz homomorphism,
    \[\Lambda:\calE_k(G) \to O(T(kG)), \, [C] \mapsto \Lambda(C) := \sum_{i\in \Z}(-1)^i [C_i],\] is surjective when $G$ is a $p$-group, giving a positive answer for a question posed by the author in \cite{SKM23}. As a corollary, we show that every $p$-permutation autoequivalence of a $p$-group (as defined by Boltje and Xu in \cite{BX07}) is the Lefschetz invariant of some splendid Rickard autoequivalence of the group algebra.

    \subsection{Surjectivity of the Lefschetz homomorphism}

    \begin{remark}
        If $G$ is a $p$-group, every permutation $kG$-module is indecomposable. In this case, we have a canonical ring isomorphism \[B(G) \xrightarrow{\sim} T(kG), [X] \mapsto [kX].\] Moreover, $(-)^*$ is the identity map on $T(kG)$, therefore $B(G)^\times \cong T(kG)^\times \cong O(T(kG))$.
    
        $B(G)^\times$ is famously difficult to describe for all finite groups. In fact, by an argument of tom Dieck in \cite{tD79}, the statement ``if $G$ has odd order, $B(G)^\times \cong C_2$'' is equivalent to the Feit-Thompson theorem.

        However, in the case for which $G$ is a $p$-group, Yal\c{c}in in \cite{Ya05} gave a complete generating set of $B(G)^\times$, which Bouc later refined to a basis in \cite{Bo07} by realizing $B^\times$ as a rational $p$-biset functor. Note the only interesting case is when $p = 2$. The biset operations on $B^\times$ are described by Bouc's \textit{generalized tensor induction}. For details, we refer the reader to \cite[Section 11.2]{Bou10}, but they will not be necessary for the scope of this paper.

        For any finite group $G$, $B(G)$ is be characterized similarly to how endotrivial complexes are via h-marks. For a $G$-set $X$, its \textit{mark at $H$} is the integer $|X^H|$. The $\Z$-linearization of this assignment yields the \textit{mark homomorphism}.

        \[m: B(G) \to B(G)^*, X\mapsto (f_X: G/H \mapsto |X|^H). \]

        The mark homomorphism is injective and full-rank, so $\Q \otimes_\Z m: \Q\otimes_\Z B(G) \to \Q\otimes_\Z B(G)^*$ is an isomorphism. However, it is rarely surjective.

        Finally, it is easy to see the image of $B(G)^\times$ consists of functions $f\in B(G)^*$ which take values only in $\pm 1$.
    \end{remark}

    There are two key insights for the proof of this statement. For these next statements, note $G$ is a $p$-group, so $B(G)^*\cong CF(G)$.

    \begin{prop}
        Let $G$ be a $p$-group. The following diagram commutes, where $\phi$ is the \textit{exponential map}, $\phi(f)(K) = (-1)^{f(K)}$.

        \begin{figure}[H]
            \centering
            \begin{tikzcd}
                \calE_k(G) \ar[r, "h"] \ar[dd, "\Lambda"] &  CF_b(G) \ar[rd, "\phi"]  \\
                & & (B(G)^*)^\times \\
                O(T(kG)) \ar[r, "\cong"] &  B(G)^\times \ar[ru, "m", hookrightarrow]
            \end{tikzcd}
            \label{fig:lefschetzcomm}
        \end{figure}
    \end{prop}
    \begin{proof}
        This is a reformulation of \cite[Proposition 4.6]{SKM23}.
    \end{proof}

    \begin{prop}
        Let $G$ be a $p$-group. Then $\phi(CF_b(G)) = m(B(G)^\times)$.
    \end{prop}
    \begin{proof}
        This is a reformulation of \cite[Proposition 5.1]{BoYa06}, which uses the result of Tornehave \cite{To84}.
    \end{proof}

    As a result, we obtain a surjective homomorphism $\gamma: CF_b(G) \to B(G)^\times$ for which the following diagram commutes.

    \begin{figure}[H]
        \centering
        \begin{tikzcd}
            \calE_k(G) \ar[r, "h"] \ar[dd, "\Lambda"] &  CF_b(G) \ar[dd, "\gamma", dotted] \ar[rd, "\phi"]  \\
            & & (B(G)^*)^\times \\
            O(T(kG)) \ar[r, "\cong"] &  B(G)^\times \ar[ru, "m", hookrightarrow]
        \end{tikzcd}
        \caption{The induced surjective homomorphism $\gamma: CF_b(G) \to B(G)^\times$.}
        \label{fig:lefschetzcomm2}
    \end{figure}

    \begin{theorem}\label{thm:surjetivityoflefschetz}
        Let $G$ be a $p$-group. $\Lambda: \calE_k(G) \to O(T(kG))$ is surjective.
    \end{theorem}
    \begin{proof}
        This follows from Figure \ref{fig:lefschetzcomm2}. Since $h$ is an isomorphism and $O(T(kG)) \cong B(G)^\times$, surjectivity of $\gamma: CF_b(G) \to B(G)^\times$ implies $\Lambda$ is surjective as well.
    \end{proof}

    \subsection{Every $p$-permutation autoequivalence of a $p$-group arises from a splendid Rickard equivalence}

    Next, we use Theorem \ref{thm:surjetivityoflefschetz} to prove that for any every $p$-permutation autoequivalence of a $p$-group $G$, $\gamma \in O(T^\Delta(kG,kG)$, there exists a splendid Rickard equivalence $\Gamma \in Ch^b({}_{kG}\triv_{kG})$ for which $\Lambda(\Gamma) = \gamma$.

    We refer the reader to \cite{BP20} or \cite{BrMi23} for exposition on $p$-permutation equivalences and splendid Rickard equivalence. We take a looser definition of splendid Rickard equivalence which allows for equivalences of direct sums of block algebras (such as group algebras). 

    \begin{definition}
        For any $\phi\in \Aut(G)$, define $\Delta_\phi(G)\leq G\times G$ to be the subgroup defined by \[\Delta_\phi(G) := \{(\phi(g), g) \in G\times G\mid g \in G \}.\] Define $\Delta_\phi(G)^{op}\leq G\times G$ to be the subgroup defined by \[\Delta_\phi(G)^{op} := \{(g, \phi(g))\in G\times G \mid g\in G\}.\] Observe $\Delta_\phi(G)^{op} \cong \Delta_{\phi\inv}(G)$.
        We have obvious group homomorphisms \[G\cong \Delta_\phi(G), \,g \mapsto (\phi(g),g),\] and \[G \cong \Delta_\phi(G)^{op} , \,g\mapsto (g, \phi(g)).\] In this way we identify $kG$-modules with $k\Delta_\phi(G)$-modules and $k\Delta_\phi(G)^{op}$-modules.
    \end{definition}

    \begin{prop}\label{prop:ppermsareorthogunits}
        Let $G$ be a $p$-group and let $\gamma \in O(T^\Delta(kG, kG))$ be a $p$-permutation equivalence. Then there exists a orthogonal unit $u \in O(T(k[\Delta_\phi G]))\cong O(T(kG))$ and group automorphism $\phi: G\xrightarrow{\sim} G$ such that $\gamma = \Ind^{G\times G}_{\Delta_{\phi}G}(u)$.
    \end{prop}
    \begin{proof}
        This follows from \cite[Theorem 1.1(e)]{BP15}, however we provide a mostly self-contained proof as well.

        $\calB\calP(\gamma)$ has a maximal $\gamma$-Brauer pair $\omega = (\Delta(G,\phi, G), e\otimes e^*)$ by \cite[Theorem 10.11]{BP20}, where $\phi : G\to G$ is a group automorphism and $e$ is the unique block of $Z(G)$, and all other maximal elements of $\calB\calP(\gamma)$ are $G\times G$-conjugate to $\omega$, every other element of $\calB\calP(\gamma)$ is $G\times G$-conjugate to some $\omega'$ satisfying $\omega' \leq \omega$. Therefore, every trivial source $(kG, kG)$-bimodule $M$ appearing in $\gamma$ has a maximal $M$-Brauer pair $\omega' \in \calB\calP(M)$ satisfying $\omega' \leq \omega$. In particular, there exists a subgroup $P \leq G$ for which $M$ has $\Delta_\phi (P)$ as a vertex. Therefore, we have
        \[\gamma = \left(\sum_{i=1}^a [\Ind^{G\times G}_{\Delta_\phi P_i} k] \right)- \left(\sum_{i=1}^b [\Ind^{G\times G}_{\Delta_\phi Q_i}k ]\right),\]
        for some $a,b \in \N$ and each $P_i, Q_j$ a subgroup of $G$. Since for any $P \leq G$,
        \[\Ind^{G\times G}_{\Delta_\phi P} = \Ind^{G\times G}_{\Delta_\phi G} \circ\Ind^{\Delta_\phi G}_{\Delta_\phi P},\]
        the above equality factors as \[\gamma = \Ind^{G\times G}_{\Delta_\phi G}\left(\sum_{i=1}^a [\Ind^{\Delta_\phi G}_{\Delta_\phi P_i} k] - \sum_{i=1}^b [\Ind^{\Delta_\phi G}_{\Delta_\phi Q_i} k]\right) =: \Ind^{G\times G}_{\Delta_\phi G} u,\] with $u \in T(k[\Delta_\phi G])$. It suffices to show $u \in O(T(k[\Delta_\phi G]))$. From here, we identify $T(kG)$ with $T(k[\Delta_\phi (G)])$ in the obvious way.

        We have that \[\left(\Ind^{G\times G}_{\Delta_\phi (G)} u \right)\cdot_{kG} \left(\Ind^{G\times G}_{\Delta_\phi G} u \right)^* = kG \in O(T^\Delta(kG,kG)).\] Since $u \in T(kG)$ and $G$ is a $p$-group, $u$ is a sum of transitive indecomposable virtual permutation modules, hence $\Ind^{G\times G}_{\Delta_\phi G} u $ is a sum of transitive indecomposable virtual permutation modules with twisted diagonal vertices. It is a straightforward verification that for any $P \leq G$ and $\varphi \in \Aut(G)$, \[\left(\Ind^{G\times G}_{\Delta_\varphi (P)}k \right)^*\cong \Ind^{G\times G}_{\Delta_{\varphi}(P)^{op}} k,\] as $(kG,kG)$-bimodules. Therefore, \[\left(\Ind^{G\times G}_{\Delta_\phi (G)} u\right)^*= \Ind^{G\times G}_{\Delta_{\phi}(G)^{op} } u \in T^\Delta(kG,kG).\]
        It follows from Bouc's extended tensor product formula \cite[Theorem 1.1]{Bo10b} that for any $H \leq G$ and $kH$-modules $M,N$ that \[\Ind^{G\times G}_{\Delta_\phi (H)} M \otimes_{kG} \Ind^{G\times G}_{\Delta_{\phi}(H)^{op} }N \cong \Ind^{G\times G}_{\Delta H}(M\otimes_k N),\]
        where on the left, $M$ is regarded as a $k\Delta_\phi H$-module and $N$ is regarded as a $k\Delta_{\phi}(H)^{op}$-module. Moreover, this isomorphism is natural in both arguments. Therefore, we have the following chain of equalities.
        \begin{align*}
            kG&=\left(\Ind^{G\times G}_{\Delta_\phi (G)} u \right)\cdot_{kG} \left(\Ind^{G\times G}_{\Delta_\phi (G)} u \right)^*\\
            &=\Ind^{G\times G}_{\Delta_\phi (G)} u\cdot_{kG}\Ind^{G\times G}_{\Delta_{\phi} (G)^{op}}u \\
            &\cong \Ind^{G\times G}_{\Delta G} (u\cdot_k u) \in O(T^\Delta(kG,kG))
        \end{align*}
        $\Ind^{G\times G}_{\Delta G}$ induces a split injective group homomorphism $T(kG) \to T^\Delta(kG,kG)$ with retraction induced by taking $1\times G$-fixed points, see \cite[Lemma 4.14]{SKM23}. Therefore, we have \[k = (kG)^{1\times G} = \Ind^{G\times G}_{\Delta G} (u\cdot_k u)^{1\times G} = u\cdot_k u \in T(kG).\] Since $u$ contains only permutation modules, as $G$ is a $p$-group, $u$ is self-dual, and we conclude $u \in O(T(kG))$ as desired.
    \end{proof}

    Therefore, every $p$-permutation autoequivalence of a $p$-group is induced twisted diagonally from an orthogonal unit. Given a $p$-permutation equivalence $\gamma = \Ind^{G\times G}_{\Delta_\phi G}u$ with $u \in O(T(kG))$, the obvious choice of corresponding splendid Rickard equivalence $\Gamma$ which satisfies $\Lambda(\Gamma) = \gamma$ should be $\Gamma = \Ind^{G\times G}_{\Delta_\phi G}C$, where $C$ is some endotrivial complex for which $\Lambda (C) = u$. Indeed, \[\Lambda \left(\Ind^{G\times G}_{\Delta_\phi G}C \right)= \Ind^{G\times G}_{\Delta_\phi G}u = \gamma \in O(T^\Delta(kG,kG)).\]

    \begin{lemma}\label{lemma:twodiffduals}
        Let $M$ be a $kG$-module. We have a natural isomorphism \[(\Ind^{G\times G}_{\Delta_{\phi}(G)} M)^* \cong (\Ind^{G\times G}_{\Delta_{\phi}(G)^{op}}M)^*,\] as $(kG,kG)$-bimodules, where the dual induced on the left-hand side arises from the left $kG$-module $k$-dual, and the dual on the right-hand side arises from the bimodule $k$-dual.
    \end{lemma}
    \begin{proof}
        We first note the bimodule structure on the left- and right-hand sides of the proposed isomorphism.

        On the left, for $a,b \in G$, $(\Ind^{G\times G}_{\Delta_{\phi}(G)} M)^*$, has actions defined by:
        \begin{align*}
            a\cdot f\big((g_1,g_2)\otimes m\big) \cdot b &= (a,b\inv) \cdot f\big((g_1,g_2)\otimes m\big)\\
            &= f\big((a\inv, b)(g_1,g_2)\otimes m\big)\\
            &= f\big((a\inv g_1,bg_2)\otimes m\big),
        \end{align*}
        and on the right, $(\Ind^{G\times G}_{\Delta_{\phi}(G)^{op}}M)^*$ corresponds to the bottom left composite, that is, it has actions defined by:
        \begin{align*}
            a\cdot f\big((g_1,g_2)\otimes m\big) \cdot b &= f\big(b\cdot \big((g_1,g_2)\otimes m\big)\cdot a\big)\\
            &= f\big((bg_1,a\inv g_2)\otimes m\big).
        \end{align*}
        We define an isomorphism $\psi: (\Ind^{G\times G}_{\Delta_{\phi}(G)} M)^* \cong (\Ind^{G\times G}_{\Delta_{\phi}(G)^{op}}M)^*$ as follows:
        \begin{align*}
            \psi: (\Ind^{G\times G}_{\Delta_{\phi}(G)} M)^* &\to (\Ind^{G\times G}_{\Delta_{\phi}(G)^{op}}M)^*\\
            f&\mapsto \big((g_1,g_2)\otimes m \mapsto f((g_2,g_1)\otimes m)\big)
        \end{align*}
        Here, $(g_2,g_1)\otimes m$ as above is considered an element in $\Ind^{G\times G}_{\Delta_{\phi}(G)} M)^*$. We first check well-definedness of the mapping with respect to the tensor product. Let $f \in (\Ind^{G\times G}_{\Delta_{\phi}(G)} M)^*$ and $(g_1,g_2)\otimes m \in \Ind^{G\times G}_{\Delta_{\phi}(G)^{op}}M$. We have for any $g \in G$ that \[(g_1,g_2)\otimes m = (g_1g\inv, g_2\phi(g\inv))\otimes gm.\]
        Then for any $f \in (\Ind^{G\times G}_{\Delta_{\phi}(G)} M)^*$,
        \begin{align*}
            \psi(f)((g_1,g_2)\otimes m)&= \psi(f)((g_1g\inv, g_2\phi(g\inv))\otimes gm)\\
            &= f((g_2\phi(g\inv), (g_1g\inv)\otimes gm)\\
            &= f((g_2, g_1)\otimes m)\\
            &=  \psi(f)((g_1,g_2)\otimes m)
        \end{align*}
        Thus the map is well-defined. We next check that it is a $(kG,kG)$-bimodule homomorphism. Let $a,b \in G$.
        \begin{align*}
            \psi(a\cdot f\cdot b)((g_1, g_2)\otimes m) &= \psi (f((a\inv, b)\cdot -))((g_1, g_2)\otimes m)\\
            &= f((a\inv g_2, bg_1)\otimes m)\\
            &= \psi(f)((bg_1, a\inv g_2)\otimes m)\\
            &=(a\cdot \psi(f)\cdot b)((g_1, g_2)\otimes m)
        \end{align*}
        Therefore, $\psi$ is a $(kG,kG)$-bimodule homomorphism. It is apparent $\psi$ is an isomorphism, as its inverse is constructed similarly. Finally we must check naturality. This follows from commutativity of the following diagram, where $f: M\to N$ is a $kG$-module homomorphism.

        \begin{figure}[H]
            \centering
            \begin{tikzcd}
                (k[G\times G]\otimes_{k\Delta_\phi(G)} M)^* \ar[r, "\psi_M"] & (k[G\times G]\otimes_{k\Delta_\phi(G)^{op}} M)^*\\
                (k[G\times G]\otimes_{k\Delta_\phi(G)} N)^* \ar[r, "\psi_N"] \ar[u, "(\id \otimes f)^*"] & (k[G\times G]\otimes_{k\Delta_\phi( G)^{op}} N)^* \ar[u, "(\id \otimes f)^*"]
            \end{tikzcd}
        \end{figure}
    \end{proof}

    \begin{theorem}\label{thm:ppermsliftforpgroups}
        Let $G$ be a $p$-group and let $\gamma \in O(T^\Delta(kG,kG))$ be a $p$-permutation autoequivalence of $kG$. There exists a splendid Rickard equivalence $\Gamma$ satisfying $\Lambda(\Gamma)=\gamma.$
    \end{theorem}
    \begin{proof}
        By Proposition \ref{prop:ppermsareorthogunits}, we have $\gamma = \Ind^{G\times G}_{\Delta_\phi(G)} u$ for some $\phi \in \Aut(G)$ and $u \in O(T(kG)) \cong O(T(k[\Delta_\phi(G)]))$, After identifying $O(T(kG))$ with $B(G)^\times$, by \ref{thm:surjetivityoflefschetz}, there exists an endotrivial complex $C$ satisfying $\Lambda(C) = u$. Moreover, by additivity of induction, \[\Ind^{G\times G}_{\Delta_\phi(G)} C = \gamma = \Ind^{G\times G}_{\Delta_\phi(G)} u.\] It suffices to show $\Ind^{G\times G}_{\Delta_\phi(G)} C $ is a splendid Rickard autoequivalence of $kG$. We have:
        \begin{align*}
            (\Ind^{G\times G}_{\Delta_\phi(G)} C)\otimes_{kG} (\Ind^{G\times G}_{\Delta_\phi(G)} C)^*
            &\cong (\Ind^{G\times G}_{\Delta_\phi(G)} C)\otimes_{kG} (\Ind^{G\times G}_{\Delta_\phi(G)^{op}} C)^*\\
            \intertext{The above isomorphism is the identification in Lemma \ref{lemma:twodiffduals}.}
            &\cong (\Ind^{G\times G}_{\Delta_\phi(G)} C)\otimes_{kG} (\Ind^{G\times G}_{\Delta_\phi(G)^{op}} C^*)\\
            \intertext{This follows from the standard natural $kG$-module isomorphism $\Ind^G_H M^* \cong (\Ind^G_H M)^*$.}
            &\cong \Ind^{G\times G}_{\Delta G} (C\otimes_k C^*)\\
            \intertext{This follows from the isomorphism used in the proof of Proposition \ref{prop:ppermsareorthogunits}.}
            &\cong kG
        \end{align*}
        Thus by \cite[Theorem 2.1]{R96}, $\Ind^{G\times G}_{\Delta_\phi(G)} C$ is a splendid Rickard equivalence.
    \end{proof}

    \textbf{Acknowledgments:} The author thanks his advisor Robert Boltje for suggesting the topic of endotrivial complexes for this paper, and for many helpful and thoughtful discussions, and for introducing to him to Erg\"{u}n Yal\c{c}in. He thanks Erg\"{u}n Yal\c{c}in for enlightening discussions which led to the main classification result. He thanks the anonymous referee for their careful reading and suggestions. He thanks Paul Balmer for helpful feedback on exposition and for a rather eye-opening discussion that has led to numerous forthcoming developments. Finally, he thanks Martin Gallauer for his observation that induction of endotrivial complexes corresponds to the topological norm map, as well as numerous other stimulating and helpful discussions throughout the year. 

    \bibliography{bib}
    \bibliographystyle{alpha}

\end{document}